\documentclass[12pt]{article}
\usepackage{amsmath, amssymb, amsthm}
\usepackage{graphicx,float,mathrsfs,enumitem}
\usepackage{xpatch,thmtools,wasysym,algorithm,algorithmic}
\usepackage{hyperref,cleveref}
\usepackage[margin=1in]{geometry}
\usepackage{multirow}
\usepackage{booktabs}
\usepackage[toc]{appendix}
\usepackage{xcolor,soul}

\usepackage[pagewise]{lineno}\nolinenumbers

\usepackage{setspace}

\newcommand{\set}[1]{\{#1\}}

\newcommand{\real}{\mathbb{R}}

\newtheorem{theorem}{Theorem}[section]
\newtheorem{lemma}{Lemma}[section]
\newtheorem{proposition}{Proposition}[section]
\newtheorem{corollary}{Corollary}[section]
\newtheorem{criteria}{Criteria}[section]

\DeclareMathOperator{\bu}{\mathbf{u}}

\DeclareMathOperator{\bn}{\mathbf{n}}
\DeclareMathOperator\NA{NA}
\DeclareMathOperator\newt{Newt}
\DeclareMathOperator\sign{sign}
\DeclareMathOperator{\gna}{\gamma\text{NAA}(\hat{r})}
\newcommand{\gnaarg}[1]{\gamma\text{NAA}(#1)}
\DeclareMathOperator\ri{Ri}

\title{Analysis of an Adaptive Safeguarded 
Newton-Anderson Algorithm of Depth One with Applications to 
Fluid Problems}
 
\date{\today}

 \author{Matt Dallas
\footnote{Corresponding author: Matt Dallas.  MD and SP were 
supported in part by the US NSF  
Project DMS 2011519 (PI: Pollock). LR was supported by 
US NSF Project DMS 2011490 (PI: Rebholz).
} $^{,1}$ 
\quad
Sara Pollock$^{2}$
\quad
Leo G. Rebholz$^{3}$\\
\small{$^1$Department of Mathematics, University of Dallas, USA
(mdallas@udallas.edu)}\\
\small{$^2$Department of Mathematics, University of Florida, 
USA (s.pollock@ufl.edu)}\\
\small{$^3$Department of Mathematics, Clemson University, USA (rebholz@clemson.edu)} 
}

\begin{document}

\maketitle

\begin{abstract}
The purpose of this paper is to develop a 
practical strategy to accelerate Newton's method in 
the vicinity of singular points.
We present an adaptive safeguarding scheme 
with a tunable parameter, which we call 
adaptive $\gamma$-safeguarding, that 
one can use in tandem with Anderson acceleration
to improve the performance of Newton's 
method when solving 
problems at or near singular points. The key features of adaptive 
$\gamma$-safeguarding are that it converges locally 
for singular problems, and it can detect nonsingular problems 
automatically, 
in which case the Newton-Anderson 
iterates are scaled towards a standard Newton step. 
The result is a flexible algorithm that performs well for singular and nonsingular problems, and can recover convergence from both standard Newton and Newton-Anderson with the 
right parameter choice.
This leads to faster local convergence compared to both 
Newton's method, and Newton-Anderson without safeguarding, 
with effectively no additional computational cost.
We demonstrate three strategies 
one can use when implementing Newton-Anderson and 
$\gamma$-safeguarded Newton-Anderson to solve parameter-dependent
problems near singular points. For our benchmark problems, we  
take two parameter-dependent incompressible flow 
systems: flow 
in a channel and Rayleigh-B\'enard convection. 
\end{abstract}

\section{Introduction}
Nonlinear systems of equations 
of the form $f(x) = 0$, with $f:\mathbb{R}^n\to\mathbb{R}^n$,  
arise frequently in applications. 
Many times the solutions depend on parameters 
(see \cite{BePi20,BoDaFa22,Ku19,pichithesis,PiQuRo20,PiStBaRo22,Ue21} 
and references therein) that can have a significant effect 
on the solution. 
Of particular interest are bifurcation points, which 
are characterized by the breakdown of 
local uniqueness 
of a solution for a particular parameter, and 
correspond to a qualitative change in the 
solution's behavior \cite{BePi20,Kiel12,pichithesis}.  
Sets of solutions of similar qualitative 
behavior are called branches \cite{pichithesis}.
Studying these branches provides a 
more complete understanding of the solutions, 
and has applications in a wide variety of fields such as 
echocardiography \cite{PiStBaRo22}, economics \cite{YuSuRe20}, physics \cite{PiQuRo20}, and engineering 
\cite{ZhQi23}. 
A necessary condition for bifurcation 
comes from 
the Implicit Function Theorem, which says that 
the Jacobian at a solution $x^*$, $f'(x^*)$, is 
necessarily singular 
if $x^*$ is a bifurcation point \cite[p.~8]{Kiel12}. 
Bifurcation points are thus examples of {\it singular points}, i.e.,
points $x$ for which $f'(x)$ is singular.
Techniques for computing solution branches such 
as continuation \cite{Ue21} and deflation
\cite{BoDaFa22,FaBiFu15} require many solves of $f(x)=0$ 
as the parameter is varied, and these problems become 
singular or nearly singular near bifurcation points. 
A popular method for solving nonlinear equations 
is Newton's method defined in Algorithm \ref{alg:newt}. 

\begin{algorithm}[H]
\begin{algorithmic}[1]
\caption{Newton}
\label{alg:newt}
\STATE{Choose $x_0\in\mathbb{R}^n$.} 
\FOR{k=1,2,...}
\STATE $w_{k+1}\gets -f'(x_k)^{-1}f(x_k)$
\STATE $x_{k+1}\gets x_k+w_{k+1}$
\ENDFOR 
\end{algorithmic}
\end{algorithm}

When $f'(x)$ is Lipschitz continuous and $f'(x^*)$ is 
nonsingular,
Newton's method exhibits local quadratic convergence in a ball 
centered at $x^*$. This is essentially the celebrated 
Newton-Kantorovich theorem \cite{Or68}.
If we remove the nonsingular assumption and let $f'(x^*)$ 
be singular, then the convergence 
behavior changes dramatically.  
Rather than local quadratic convergence from any $x_0$ in a 
sufficiently small ball around $x^*$, we see local linear convergence 
in a starlike domain of convergence  
around $x^*$ \cite{DeKeKe83,DeKe80,Gr80,Re78,Re79}. Since bifurcation points are necessarily 
singular points, this means 
that continuation or deflation algorithms using Newton's method 
may converge slowly or fail to converge at or near a bifurcation
point. This challenge has motivated the study of modifications 
\cite{DeKeKe83,DeKe82,FiIzSo21-2,FiIzSo21-1,Gr85,Gr80,HMT09,KeSu83} 
or alternatives 
\cite{BeFi12,BFHIY14,BelMo15,FaZe13,KaYaFu04,ScFr84} to Newton's 
method that can improve convergence behavior at singular points. 
Among the modifications, Richardson extrapolation and overrelaxation 
are popular and perform well. They can achieve superlinear and 
arbitrarily fast linear convergence respectively under certain 
conditions \cite{Gr85} at the cost of additional 
function evaluations and some knowledge of 
the order of the singularity \cite{Gr80,Gr85}. The order may be inferred from 
monitoring the singular values of the Jacobian as the solve 
progresses. A popular alternative to Newton's 
method for singular problems is the Levenberg-Marquardt method 
\cite{BeFi12,BFHIY14,BelMo15,FaZe13,KaYaFu04}. Under standard 
assumptions and the 
local error bound, or local Lipschitzian error bound, the 
Levenberg-Marquardt method can achieve local quadratic 
convergence \cite{KaYaFu04}. The local error bound is known to be much weaker 
than the standard nonsingularity assumption. Indeed, it can hold 
even for singular problems \cite{BFHIY14}. However, without 
the local error bound or nonsingularity, 
Levenberg-Marquardt is not guaranteed
such success. In the absence of the local error bound, one 
may insist that the function $f$ is 2-regular 
\cite{FiIzSo21-2,FiIzSo21-1,IzKuSo18-1,IzKuSo18-2}, in which case 
Levenberg-Marquardt converges locally linearly in a starlike domain 
much like Newton's method \cite{IzKuSo18-1}. 

The method of interest
in this paper, Anderson acceleration, has a long track record of 
accelerating linearly converging fixed-point methods, and 
has been applied in many different fields 
\cite{AJW17,LWWY12,PaMa22,sim-app-2018,PRX18,TRL04,WHdS21}.
Further, when applied as a modification
to Newton's method at singular points,  
in contrast to Richardson extrapolation or overrelaxation 
discussed in \cite{Gr85}, this success requires 
 no knowledge of the order of the root or additional function 
evaluations. It was also shown recently that under a condition 
equivalent to 2-regularity (discussed further in Section 
\ref{sec:gsg}), Anderson accelerates Newton's 
method when applied to singular problems \cite{DaPo23}, 
and therefore outperforms Levenberg-Marquardt in 
the absence of the local error bound. This was demonstrated  
numerically in \cite{DaPo23}.
Also known as 
Anderson extrapolation or Anderson mixing \cite{WaNi11}, Anderson 
acceleration was first introduced in 1965 by 
D.G. Anderson in \cite{Anderson65} to improve the convergence of fixed-point iterations applied to 
integral equations. The algorithm combines the previous $m+1$ 
iterates and update steps into a 
new iterate at each step of the solve. The number $m$ is 
commonly known as the 
algorithmic depth. The combination of the $m+1$ iterates often 
involves solving a least-squares problem, 
but since $m$ is typically small, the computational cost of this step is in general orders of magnitude less than that of a single linear solve.
There are problems for which taking $m$ much larger can be beneficial \cite{PoRe21,WaNi11}, and later in this paper we will see numerically 
that increasing
$m$ can improve or recover convergence near bifurcation points. 
With greater depths, the least-squares problem 
may suffer from ill-conditioning if proper care 
is not taken in the implementation \cite{PoRe23}.
In \cite{DaPo23}, the authors developed a convergence and 
acceleration theory for Anderson accelerated 
Newton's method with depth $m=1$, defined
in Algorithm \ref{alg:na1}, 
applied to singular problems. This is a special case of depth 
$m\geq 1$ given in Algorithm \ref{alg:nam}. 

\begin{algorithm}[H]
\begin{algorithmic}[1]
\caption{Newton-Anderson(1)}
\label{alg:na1}
\STATE{Choose $x_0\in\mathbb{R}^n$. Set $w_1=-f'(x_0)^{-1}f(x_0)$, and $x_1=x_0+w_1$.}
\FOR{k=1,2,...}
\STATE $w_{k+1}\gets -f'(x_k)^{-1}f(x_k)$
\STATE $\gamma_{k+1}\gets (w_{k+1}-w_k)^Tw_{k+1}/\|w_{k+1}-w_k\|_2^2$
\STATE $x_{k+1}\gets x_k+w_{k+1}-\gamma_{k+1}(x_k-x_{k-1}+w_{k+1}-w_k)$
\ENDFOR 
\end{algorithmic}
\end{algorithm}

\begin{algorithm}[H]
\begin{algorithmic}[1]
\caption{Newton-Anderson(m)}
\label{alg:nam}
\STATE{Choose $x_0\in\mathbb{R}^n$ and $m\geq 0$. 
Set $w_1=-f'(x_0)^{-1}f(x_0)$, and $x_1=x_0+w_1$.}
\FOR{k=1,2,...}
\STATE $m_k\gets \min\{k,m\}$
\STATE $w_{k+1}\gets -f'(x_k)^{-1}f(x_k)$
\STATE $F_k = \big( (w_{k+1}-w_k)\cdots (w_{k-m+2}-w_{k-m+1})\big)$
\STATE $E_k = \big( (x_k-x_{k-1})\cdots (x_{k-m+1}-x_{k-m})\big)$
\STATE $\gamma_{k+1}\gets \text{argmin}_{\gamma\in\mathbb{R}^m} \|w_{k+1}-F_k\gamma\|_2^2$
\STATE $x_{k+1}\gets x_k+w_{k+1}-(E_k+F_k)\gamma_{k+1}$
\ENDFOR 
\end{algorithmic}
\end{algorithm}

A major challenge when proving 
convergence of Newton-like methods near singular points 
is ensuring that the iterates remain well-defined. 
The authors in \cite{DaPo23} introduced a novel 
safeguarding scheme called 
$\gamma$-safeguarding, defined in Algorithm \ref{alg:gsg-og} 
in Section \ref{sec:gsg}, 
to deal with this problem. 
The result was a convergence proof 
for $\gamma$-safeguarded Newton-Anderson,
and 
it was 
observed numerically to 
perform better or no worse
compared to standard 
Newton-Anderson, particularly when applied to nonsingular
problems. 

The purpose of this paper is to extend these 
ideas of
\cite{DaPo23} by developing 
an adaptive version of 
$\gamma$-safeguarding 
that automatically 
detects nonsingular problems, and to demonstrate the effectiveness 
of Newton-Anderson and adaptive 
$\gamma$-safeguarded Newton-Anderson at solving parameter-dependent 
PDEs near bifurcation points in fluid problems.
This new adaptive scheme is 
proven to be locally convergent under the 
same conditions as the non-adaptive scheme,
and the automatic detection of nonsingular problems
enables local quadratic 
convergence 
when applied to nonsingular problems.
Such a property is desirable 
when solving nonlinear problems {\it near} bifurcation points, because even if the problem 
itself is not singular, convergence can still be affected if it is close to a singular 
problem \cite{DeKe85}.
One would like to enjoy the benefits of 
Newton-Anderson in the preasymptotic regime such as 
a larger domain of convergence \cite{posc20}, but not 
lose quadratic convergence in the 
asymptotic regime if the problem is nonsingular. It 
is not always known a priori if a problem is 
singular or nonsingular, and 
Anderson acceleration can reduce the order of convergence 
when applied to superlinearly converging iterations 
such as Newton's method applied to a nonsingular
problem \cite{ReXi23}. This problem is solved with adaptive 
$\gamma$-safeguarded Newton-Anderson  
with effectively no additional computational cost. 
We also show numerically
that increasing the algorithmic depth 
of the Newton-Anderson algorithm can recover convergence 
when Newton fails
near bifurcation 
points, but only for specific choices of $m$.

The algorithms of interest in this paper are 
Newton, defined in Algorithm \ref{alg:newt}, Newton-Anderson 
with algorithmic depth $1$ \big(NA($1$)\big)  
defined in Algorithm \ref{alg:na1},
Newton-Anderson with algorithmic depth $m$ \big(NA($m$)\big)  
defined in Algorithm \ref{alg:nam}, 
$\gamma$-safeguarded Newton-Anderson \big($\gamma$NA($r$)\big), defined in Algorithm \ref{alg:gsgna}, and adaptive $\gamma$-safeguarded Newton-Anderson \big($\gna$\big), defined in 
Algorithm \ref{alg:adapgsgna}.
We implement $\gamma$NA($r$) and $\gna$ by replacing line 
5 in Algorithm \ref{alg:na1} with, respectively, $\gamma$-safeguarding (Algorithm \ref{alg:gsg-og}) and 
{\it adaptive} $\gamma$-safeguarding (Algorithm \ref{alg:adapgsg}).
The norm in the algorithms is the Euclidean norm. 
There should be no confusion
between NA($m$) and $\gamma$NA($r$) or $\gna$ since $\gamma$-safeguarding is currently 
only developed for depth $m=1$. The rest of the paper 
is organized as follows. In Section \ref{sec:gsg}, we 
review the original $\gamma$-safeguarding algorithm and its 
role in the convergence theory developed in \cite{DaPo23}. 
In Section \ref{sec:adapgsg}, 
we introduce the new adaptive $\gamma$-safeguarding algorithm 
and prove that $\gna$ can recover local quadratic 
convergence when applied to nonsingular problems 
in Corollary \ref{cor:localquad}.
We conclude in Section 
\ref{sec:numerics} by applying NA and $\gna$ to two 
parameter-dependent incompressible flow systems, and discussing 
various strategies one can use when employing $\gamma$-safeguarding.   

\section{The $\gamma$-safeguarding algorithm}
\label{sec:gsg}

For the theory discussed in Section 2 and Section 3, 
unless stated otherwise, we take 
$f:\real^n\to\real^n$ to be a $C^3$ function, 
$f(x^*) = 0$, $N = \text{null}\left(f'(x^*)\right)$, $R = \text{range} 
\left(f'(x^*)\right)$, $\dim N = 1$, and $P_N$ and 
$P_R$ to be the 
orthogonal projections onto $N$ and $R$ respectively.
The assumption that $R\perp N$ is common in the 
singular Newton literature 
\cite{DeKeKe83,griewank-thesis}, and no generality 
is lost in finite dimensions.
Indeed, Newton's method is essentially 
invariant under nonsingular affine transformations of the 
domain and nonsingular linear transformations of the range \cite{deuf05}. Thus to determine the convergence behavior 
of Newton's method applied to a general $C^3$ 
function $F:\mathbb{R}^n\to\mathbb{R}^n$, it suffices to study 
that of $f(x) = U^TF(x)V$, where $U$ and $V$ come from the 
single value decomposition $F'(x^*) = U\Sigma V^T$. Since 
$f'(x^*) = U^TF'(x^*)V = \Sigma$, it follows that 
$\text{null}\left(f'(x^*)\right)\perp \text{range}\left(f'(x^*)\right)$. 
Lastly, let $\|\cdot\|$ denote the Euclidean 2-norm, $B_{\rho}(x)$ denote a ball of radius $\rho$ centered at $x$, 
$e_k = x_k-x^*$, $w_{k+1} = -f'(x_k)^{-1}f(x_k)$, 
 and 
\begin{align}\label{eq:opgain}
\theta_{k+1} = \frac{\|w_{k+1}-\gamma_{k+1}(w_{k+1}-w_k)\|}{\|w_{k+1}\|},
\end{align}
where $\gamma_{k+1}$ is computed via Algorithm \ref{alg:na1}. 
The term $\theta_{k+1}$ 
is known as the {\it optimization gain},
and is key to determining when Anderson acceleration 
is successful both in the singular and nonsingular 
cases \cite{DaPo23,PoRe21}.

\begin{algorithm}[H] \caption{$\gamma$-safeguarding} \label{alg:gsg-og}
\begin{algorithmic}[1]
\STATE {Given $x_k$, $x_{k-1}$, $w_{k+1}$, $w_k$, $\gamma_{k+1}$, and $r\in(0,1)$. Set $\lambda=1$.
\STATE $\beta_{k+1} \gets r\|w_{k+1}\|/\|w_k\|$}
\IF{$\gamma_{k+1}=0$ {\bf or } $\gamma_{k+1}\geq 1$}
\STATE $\lambda \gets 0$
\ELSIF {$|\gamma_{k+1}|/|1-\gamma_{k+1}|>\beta_{k+1}$} 
\STATE $\lambda \gets \dfrac{\beta_{k+1}}{\gamma_{k+1}\left(\beta_{k+1}+\sign(\gamma_{k+1})\right)}$
\ENDIF 
\STATE $x_{k+1}\gets  x_k+w_{k+1}-\lambda\gamma_{k+1}(x_k-x_{k-1}+w_{k+1}-w_k)$
\end{algorithmic}
\end{algorithm}

\begin{algorithm}[H]
\begin{algorithmic}[1]
\caption{$\gamma$-Safeguarded Newton-Anderson $\left(\gamma\text{NA}(r)\right)$}
\label{alg:gsgna}
\STATE{Choose $x_0\in\mathbb{R}^n$ and $r\in(0,1)$. Set $w_1=-f'(x_0)^{-1}f(x_0)$, and $x_1=x_0+w_1$}
\FOR{k=1,2,...}
\STATE $w_{k+1}\gets -f'(x_k)^{-1}f(x_k)$
\STATE $\gamma_{k+1}\gets (w_{k+1}-w_k)^Tw_{k+1}/\|w_{k+1}-w_k\|_2^2$
\STATE {$\beta_{k+1} \gets r\|w_{k+1}\|/\|w_k\|$}
\STATE $\lambda \gets 1$
\IF{$\gamma_{k+1}=0$ {\bf or } $\gamma_{k+1}\geq 1$}
\STATE $\lambda \gets 0$
\ELSIF {$|\gamma_{k+1}|/|1-\gamma_{k+1}|>\beta_{k+1}$} 
\STATE $\lambda \gets \dfrac{\beta_{k+1}}{\gamma_{k+1}\left(\beta_{k+1}+\sign(\gamma_{k+1})\right)}$
\ENDIF 
\STATE $x_{k+1}\gets  x_k+w_{k+1}-\lambda\gamma_{k+1}(x_k-x_{k-1}+w_{k+1}-w_k)$
\ENDFOR 
\end{algorithmic}
\end{algorithm}

To implement 
$\gamma$NA$(r)$, one replaces line 5 in Algorithm \ref{alg:na1} 
with Algorithm \ref{alg:gsg-og} giving Algorithm 
\ref{alg:gsgna}. 
The resulting steps, $x_{k+1} = x_k+w_{k+1}-\lambda\gamma_{k+1}(x_k-x_{k-1}+w_{k+1}-w_k)$, can be viewed as 
NA steps scaled by $\lambda$ towards a Newton step based on a 
user-chosen parameter, $r$, which is set at the start of the solve. This parameter determines 
how strongly the NA steps are scaled towards a Newton step, 
i.e., how close a $\gamma$NA$(r)$ step is to a 
Newton step.
When $r\approx 0$, 
the $\gamma$NA$(r)$ iterates will be heavily scaled towards the latest 
Newton step, and when $r\approx 1$, the $\gamma$NA$(r)$ 
iterates will behave more like 
standard NA. The idea behind 
$\gamma$-safeguarding is to take advantage of the particular 
convergence behavior of Newton's method near singular points, which
we now describe.  
Since $\dim N = 1$, $N$ is spanned by 
some nonzero $\varphi\in\mathbb{R}^n$. If the linear operator 
$\hat{D}(\cdot):= P_Nf''(x^*)\left(\varphi,P_N(\cdot)\right)$
is nonsingular as a map from $N$ to $N$, 
then there exists 
$\hat{\rho}>0$ and $\hat{\sigma}>0$ such that $f'(x)$ 
is nonsingular for all $x\in\hat{W}:= B_{\hat{\rho}}(x^*)
\cap \{ x : \|P_R(x-x^*)\| < \hat{\sigma} \|P_N(x-x^*)\|\}$ and Newton's method 
converges linearly to $x^*$ from any $x_0\in\hat{W}$ 
\cite{DeKeKe83}. 
The assumption that $\hat{D}$ is nonsingular as a linear map on 
$N$ is equivalent to the assumption that $f$ is 2-regular at $x^*$ 
in the direction 
$\varphi$ \cite{FiIzSo21-1}. A stronger result due to Griewank 
\cite[Theorem 6.1]{Gr80} says that when $\hat{D}$ is 
nonsingular, Newton's method converges from every 
$x_0$ in a starlike region with density 1 with respect to $x^*$, 
and the iterates lead into $\hat{W}$ provided 
$\|x_0-x^*\|$ is sufficiently small. Thus for our purposes in this 
paper studying local convergence of NA, it suffices to study 
the behavior of iterates in $\hat{W}$. Though the focus of this work is the $\dim N = 1$ case, we note that Griewank's Theorem \cite[Theorem 6.1]{Gr80} holds for 
$\dim N > 1$, and numerical experiments from \cite{DaPo23} on 
small-scale problems indicate that NA and $\gna$ are effective when $\dim N > 1$. While the current theoretical results 
for $\gna$ require $\dim N = 1$, extensions of these results for $\dim N > 1$ will be studied in future work.

The main challenge in accelerating Newton's method is ensuring
the iterates remain in $\hat{W}$, which requires 
$\|P_R(x-x^*)\|/\|P_N(x-x^*)\|$ to remain bounded. In 
other words, the iterates can't be accelerated ``too much" 
along the null space. 
Evidently, Anderson may accelerate Newton significantly,
especially in $\hat{W}$ when $\dim N = 1$.
This is demonstrated by the following proposition. We will use 
the notation $(x_{k}+w_{k+1})^{\alpha} := x_k+w_{k+1}-\gamma_{k+1}(x_k-x_{k-1}+w_{k+1}-w_k)$. Similarly, $P_Ne_k^{\alpha} = P_Ne_k-\gamma_{k+1}P_N(e_k-e_{k-1})$, $(T_kP_Re_k)^{\alpha} := T_kP_Re_k - 
\gamma_{k+1}(T_kP_Re_k - T_{k-1}P_Re_{k-1})$, and 
$w_{k+1}^{\alpha} = w_{k+1} - \gamma_{k+1}(w_{k+1}-w_k)$. Here 
$T_k$ denotes a linear map defined in the proof of Proposition \ref{prop:localquad}. We also note that 
$\hat{D}$ is nonsingular if and only if $\hat{D}(x)(\cdot):=P_Nf''(x^*)(P_N(x-x^*),P_N(\cdot))$ is nonsingular for all $x$ 
with $P_N(x-x^*)\neq 0$. 
\begin{proposition}\label{prop:localquad}
Let $f\in C^3$, $\dim N=1$, $\hat{D}$ nonsingular, and $x_k,x_{k-1}\in \hat{W}$ so that 
$x_{k+1} = (x_k+w_{k+1})^{\alpha}$ 
is well-defined. If $P_Nw_{k+1}\neq P_Nw_k$ and $|
1-\gamma_{k+1}|\,\|P_Ne_k\|\neq |\gamma_{k+1}|\,\|P_Ne_{k-1}\|\neq 0$,
then for sufficiently small $\hat{\sigma}>0$ and $\hat{\rho}>0$
there is a constant $C=C(\hat{\sigma},\hat{\rho})$ such that 
\begin{align}\label{eq:quadbound}
\|P_Ne_{k+1}\| \leq C\max\{|1-\hat{\gamma}|,|\hat{\gamma}|\}\max\{\|e_k\|^2,\|e_{k-1}\|^2\},
\end{align}
where $\hat{\gamma}:=(P_Nw_{k+1})^T(P_Nw_{k+1}-P_Nw_k)/\|P_Nw_{k+1}-P_Nw_k\|^2$.
\end{proposition}

\begin{proof}
First note that $w_{k+1}^{\alpha} = 
P_Nw_{k+1}^{\alpha} + P_Rw_{k+1}^{\alpha}$. 
Since $\gamma_{k+1}:=\text{argmin}_{\gamma\in\real} \|w_{k+1}-
\gamma(w_{k+1}-w_k)\|$ by Algorithm \ref{alg:na1}, and $R \perp N$, 
we have that for any $\gamma\in \real$, 
\begin{align}\label{eq:wquadbound}
\|w_{k+1}^{\alpha}\|^2 \leq  \|P_Nw_{k+1}-\gamma(P_Nw_{k+1}-P_Nw_k)\|^2 
+ \|P_Rw_{k+1}-\gamma(P_Rw_{k+1}-P_Rw_k)\|^2. 
\end{align}
Taking $\gamma=\hat{\gamma}$ gives $\|P_Nw_{k+1}-\hat{\gamma}(P_Nw_{k+1}-P_Nw_k)\|^2=0$. 
By Proposition 3.1 in \cite{DaPo23}, we can write 
\begin{align}\label{eq:wexpansion}
w_{k+1}^{\alpha}=-(1/2)P_Ne_k^{\alpha}+\left((T_k-I)P_Re_k\right)^{\alpha}+q_{k-1}^k,
\end{align}
 where 
$T_k(\cdot) := (1/2)\hat{D}(x_k)^{-1}f''(x_k)(e_k,\cdot)$ is a linear 
map whose range lies in $N$ \cite{DaPo23}, and $\|q_{k-1}^k\| \leq
c\max\{|1-\gamma_{k+1}|,|\gamma_{k+1}|\}
\max\{\|e_k\|^2,\|e_{k-1}\|^2\}$ for 
a constant $c$ determined by $f$. Hence 
$\|w_{k+1}^{\alpha}\|\leq \|P_Rw_{k+1}-\hat{\gamma}P_R(w_{k+1}-w_k)\|
\leq c\max\{|1-\hat{\gamma}|,|\hat{\gamma}|\}\max\{\|e_k\|^2,\|e_{k-1}\|^2\}.$
We also have from Proposition 3.1 in \cite{DaPo23} that 
$P_Ne_{k+1} = (1/2)P_Ne_k^{\alpha}+(T_kP_Re_k)^{\alpha}+
q_{k-1}^k$. Let $\mu_{k+1}^e = \|(T_kP_Re_k)^{\alpha}+
q_{k-1}^k\|/\|(1/2)P_Ne_k^{\alpha}\|$. This expansion of 
$P_Ne_{k+1}$ combined with 
Equation \eqref{eq:wexpansion} gives
\begin{align}
\|P_Ne_{k+1}\| \leq \left(\frac{1+\mu_{k_+1}^e}{1-\mu_{k+1}^e}\right)
\|P_Nw_{k+1}^{\alpha}\|\leq C_k\max\{|1-\hat{\gamma}|,|\hat{\gamma}|\}\max\{\|e_k\|^2,\|e_{k-1}\|^2\},
\end{align}
where we define $C_k := c(1+\mu_{k_+1}^e)(1-\mu_{k+1}^e)^{-1}
$. One can show that $\mu_{k+1}^e\to 0$ as $\hat{\sigma}$ and 
$\hat{\rho}$ tend to zero \cite{DaPo23}. Thus for sufficiently small $\hat{\sigma}$ and $\hat{\rho}$ we have 
$C_k \leq C$. This completes the proof.
\end{proof}

Note that $\hat{\gamma}$ can be very large when $P_Nw_{k+1} \approx P_Nw_k$, 
which may correspond to the terminal phase of the solve. So this bound is meaningful in 
the asymptotic regime when there is still a significant decrease in the residual at each step. 
Such a bound is good for a single step, 
but this dramatic acceleration 
of $P_Ne_{k+1}$ could place $x_{k+1}$ outside the 
domain of invertibility $\hat{W}$, i.e., $x_{k+1}$ may 
stray too far from $N$. 
This is where $\gamma$-safeguarding is useful. 
It ensures
that the $\gamma$NA$(r)$ iterates remain within $\hat{W}$
by taking advantage of the way Newton steps 
are attracted to $N$ and scaling NA steps towards
Newton steps when the conditions of Algorithm \ref{alg:gsg-og} 
are met. This is also the key to the convergence proof of 
$\gamma$NA$(r)$ given in \cite{DaPo23}. 
However, given the results of \cite{ReXi23}, if the problem is 
nonsingular one should use Newton's method without 
Anderson, but it is not always obvious a priori if the 
problem at hand is singular or nonsingular. 
In 
the next section, we develop an adaptive version 
of $\gamma$NA$(r)$ 
that enjoys guaranteed local convergence when 
applied to singular problems, but can also detect nonsingular 
problems automatically, at no additional computational 
cost, and ``turn off" NA in response. 
This leads to local quadratic convergence if the problem is 
nonsingular. 

\section{Adaptive $\gamma$-safeguarding}
\label{sec:adapgsg}

It was observed in \cite{DaPo23} that 
$\gamma$NA$(r)$ performed competitively 
with standard NA(1). For example, $\gamma$NA(0.5) 
could outperform NA(1) when
applied to certain nonsingular problems. This is not surprising given 
the results of \cite{ReXi23}, and it had previously been 
observed numerically in \cite{posc20} that 
NA does not necessarily improve convergence when applied to nonsingular problems. 
Thus, by setting the 
$\gamma$-safeguarding parameter $r=0.5$, thereby scaling the 
iterates 
closer to Newton steps, we see less of the effect of a full NA step on the 
order of convergence. 

Even if it is known that the problem is nonsingular, one 
may still wish to use NA to take advantage of the larger 
domain of convergence \cite{posc20}.
In such a scenario one could set $r$ close to zero, but this means 
the effect of NA is 
never completely eliminated which could lead to a smaller order of convergence 
relative to Newton for nonsingular problems. 
This, and the fact that often it is not known a priori 
if the problem is singular or nonsingular, motivates the development of an 
adaptive form of $\gamma$NA($r$) that can automatically
detect nonsingular problems, and scale an NA step accordingly,
without sacrificing local convergence and 
acceleration for singular 
problems. We will denote this 
adaptive choice of $r$ from Algorithm \ref{alg:gsg-og} by $r_{k+1}$, with $k$ denoting the iteration count. 
There are three criteria that the choice of $r_{k+1}$ 
should satisfy, which we record in Criteria \ref{criteria}.

\begin{criteria}\label{criteria}
An adaptive $\gamma$-safeguarding tolerance $r_{k+1}$ should satisfy the following.
\begin{enumerate}[ref=\theenumi{}]
\item $r_{k+1} {\ll} 1$ if $\|P_Ne_k\|/\|P_Ne_{k-1}\| {\ll} 1$;
\label{crit1}
\item $r_{k+1} \approx 1$ if $\|P_Ne_k\|/\|P_Ne_{k-1}\| \approx 1$; and 
\label{crit2}
\item $\lim_{k\to\infty} r_{k+1} = 0$ if $f'(x^*)$ is nonsingular.
\label{crit3}
\end{enumerate}
\end{criteria}

Criterion 3.1.\ref{crit1} says that if $\|P_Ne_k\|/\|P_Ne_{k-1}\|$ is very small, then 
we want to scale the NA step generated from $x_k$ and $x_{k-1}$ heavily 
towards $x_k+w_{k+1}$. In the singular case, this will (locally) keep $x_{k+1}$ 
within the domain of invertibility. Alternatively, if the problem is nonsingular, but close to a singular problem,
then scaling NA towards a Newton step is also preferred when
$\|P_Ne_k\|/\|P_Ne_{k-1}\|$  is small since then we do not
slow Newton's fast local quadratic convergence. 
Criterion 3.1.\ref{crit2} says 
that if $\|P_Ne_k\|/\|P_Ne_{k-1}\|$
is close to one, then the error is not decreasing 
significantly, and we want to allow 
NA to act on $x_k$ and $x_{k-1}$ without significant scaling of $\gamma_{k+1}$ from safeguarding.
Lastly, Criterion 3.1.\ref{crit3} is important because if $f'(x^*)$ is nonsingular, 
Newton's method will converge quadratically in a neighborhood of $x^*$. We therefore 
want to ``turn off" NA near $x^*$, and insisting that 
$r_{k+1} \to 0$ asymptotically achieves this. 

Observing Criteria 3.1.\ref{crit1}-\ref{crit3}, one may note that we essentially want $r_{k+1}$ to 
behave like 
$\|P_Ne_k\|/\|P_Ne_{k-1}\|$ within the domain of convergence $\hat{W}$.
Of course, we can not compute $\|P_Ne_k\|/\|P_Ne_{k-1}\|$, but 
Equation \eqref{eq:wexpansion} says that in $\hat{W}$,
$w_{k+1}\approx P_Ne_k$. So 
if we take $r_{k+1} = \|w_{k+1}\|/\|w_k\|$, we can expect 
Criteria 3.1.\ref{crit1}-\ref{crit2} to 
be enforced locally. For Criterion 3.1.\ref{crit3}, if $f'(x^*)$ is nonsingular, then locally we will
have $\|w_{k+1}\|/\|w_k\| \to 0$. Thus $r_{k+1}=\|w_{k+1}\|/\|w_k\|$ satisfies the 
three criteria within the domain of convergence. With 
this choice of $r_{k+1}$, we have adaptive $\gamma$-safeguarding and $\gna$. 

\begin{algorithm}[H] \caption{Adaptive $\gamma$-safeguarding} \label{alg:adapgsg}
\begin{algorithmic}[1]
\STATE {Given $x_k$, $x_{k-1}$, $w_{k+1}$, $w_k$, $\gamma_{k+1}$, and $\hat{r}\in (0,1)$, set 
$\eta_{k+1} = \|w_{k+1}\|/\|w_k\|$, $r_{k+1} = \min\{\eta_{k+1},\hat{r}\}$, and $\lambda^a=1$.}
\STATE {$\beta_{k+1} \gets r_{k+1}\eta_{k+1}$}
\IF{$\gamma_{k+1}=0$ {\bf or } $\gamma_{k+1}\geq 1$}
\STATE $\lambda^a \gets 0$
\ELSIF {$|\gamma_{k+1}|/|1-\gamma_{k+1}|>\beta_{k+1}$} 
\STATE $\lambda^a \gets \dfrac{\beta_{k+1}}{\gamma_{k+1}\left(\beta_{k+1}+\sign(\gamma_{k+1})\right)}$
\ENDIF 
\STATE $x_{k+1}\gets  x_k+w_{k+1}-\lambda^a\gamma_{k+1}(x_k-x_{k-1}+w_{k+1}-w_k)$
\end{algorithmic}
\end{algorithm}

\begin{algorithm}[H]
\begin{algorithmic}[1]
\caption{Adaptive $\gamma$-Safeguarded Newton-Anderson 
$\left(\gna)\right)$}
\label{alg:adapgsgna}
\STATE{Choose $x_0\in\mathbb{R}^n$ and $\hat{r}\in(0,1)$. Set $w_1=-f'(x_0)^{-1}f(x_0)$ and $x_1=x_0+w_1$}
\FOR{k=1,2,...}
\STATE {$w_{k+1}\gets -f'(x_k)^{-1}f(x_k)$}
\STATE {$\gamma_{k+1}\gets (w_{k+1}-w_k)^Tw_{k+1}/\|w_{k+1}-w_k\|_2^2$}
\STATE {$\eta_{k+1} \gets \|w_{k+1}\|/\|w_k\|$}
\STATE {$r_{k+1} \gets \min\{\eta_{k+1},\hat{r}\}$}
\STATE {$\beta_{k+1} \gets r_{k+1}\eta_{k+1}$}
\STATE $\lambda^a \gets 1$
\IF{$\gamma_{k+1}=0$ {\bf or } $\gamma_{k+1}\geq 1$}
\STATE $\lambda^a \gets 0$
\ELSIF {$|\gamma_{k+1}|/|1-\gamma_{k+1}|>\beta_{k+1}$} 
\STATE $\lambda^a \gets \dfrac{\beta_{k+1}}{\gamma_{k+1}\left(\beta_{k+1}+\sign(\gamma_{k+1})\right)}$
\ENDIF 
\STATE $x_{k+1}\gets  x_k+w_{k+1}-\lambda^a\gamma_{k+1}(x_k-x_{k-1}+w_{k+1}-w_k)$
\ENDFOR 
\end{algorithmic}
\end{algorithm}

Adaptive 
$\gamma$-safeguarding differs from 
Algorithm \ref{alg:gsg-og} only in line 2. In Algorithm \ref{alg:gsg-og}, $\beta_{k+1} = 
r\eta_{k+1}$ whereas $\beta_{k+1}=r_{k+1}\eta_{k+1}$ in Algorithm 
\ref{alg:adapgsg} with $r_{k+1} = \min\{\eta_{k+1},\hat{r}\}$. 
This one 
change can have a significant impact on convergence as demonstrated 
in Section \ref{sec:numerics}, and can enable locally quadratic 
convergence when applied to nonsingular problems (see Corollary 
\ref{cor:localquad}). Similar to $\gamma$NA$(r)$, 
one implements $\gna$ by replacing line 5 in 
Algorithm \ref{alg:na1} with Algorithm \ref{alg:adapgsg} 
and setting $\hat{r}$ at the start of the solve. The 
result is Algorithm \ref{alg:adapgsgna}.
The choice of $\hat{r}$ here sets the weakest safeguarding the user wants to impose. 
Thus we are always safeguarding at least as strictly as standard $\gamma$-safeguarding 
with $r=\hat{r}$. Stated concisely, we have $r_{k+1}\leq \hat{r}$.
Local convergence of $\gna$ then follows from 
Theorem $6.1$ in \cite{DaPo23}. To state this precisely, let $\lambda_{k+1}$ 
be the value of $\lambda^a$ computed by Algorithm \ref{alg:adapgsg} at step $k$, 
$\theta_{k+1}^{\lambda}=\|w_{k+1}-\lambda_{k+1}\gamma_{k+1}(w_{k+1}-w_k)\|/\|w_{k+1}\|$, $x_{k+1}^{NA} := x_k + w_{k+1} - \lambda_{k+1}\gamma_{k+1}(x_k-x_{k-1}+w_{k+1}-w_k)$, and 
$\sigma_k := \|P_Re_k\|/\|P_Ne_k\|$.
Then we have the following theorem.

\begin{theorem}\label{thm:convergence-chap3}
Let $\dim N=1$, and let $\hat{D}$ be invertible as a map on $N$.
Let $W_k:=B_{\|e_k\|}(x^*)\cap 
\{x : \|P_R(x-x^*)\|<\sigma_k\|P_N(x-x^*)\|\}$. If $x_0$ is 
chosen so that $\sigma_0<\hat{\sigma}$ and $\|e_0\|<\hat{\rho}$, for 
sufficiently small $\hat{\sigma}$ and $\hat{\rho}$, $x_1=x_0+w_1$, and 
$x_{k+1}=x_{k+1}^{NA}$ for $k\geq 1$, 
then $W_{k+1}\subset W_0$ for all $k\geq 0$ and $x_k\to x^*$. 
That is, 
$\set{x_k}$ remains well-defined and converges to $x^*$. Furthermore, 
there exist constants $C>0$ and $\kappa\in(1/2,1)$ such that
\begin{align}
\|P_Re_{k+1}\|&\leq
C\max\set{|1-\lambda_{k+1}\gamma_{k+1}|\,\|e_k\|^2,
|\lambda_{k+1}\gamma_{k+1}|\,\|e_{k-1}\|^2}\label{range-bd-main-chap3} \\
\|P_Ne_{k+1}\|&\leq \kappa\theta_{k+1}^{\lambda}\|P_Ne_{k}\| \label{null-rate-of-convergence-main-chap3}
\end{align}
for all $k\geq 1$. 
\end{theorem}

Under the assumptions of Theorem \ref{thm:convergence-chap3}, Griewank's Theorem \cite[Theorem 6.1]{Gr80} says that Newton's method almost surely leads into the domain of 
convergence $\hat{W}$ provided $x_0$ is sufficiently close to $x^*$. This effectively means that if the sequence $x_k$ generated by $\gna$ approaches $x^*$ we will have almost sure convergence eventually, and this convergence will be faster than Newton. The precise improvement is determined by the asymptotic behavior of $\theta_{k+1}^{\lambda}$. Globalization techniques such as linesearch methods may be used to bring the $\gna$ iterates 
closer to $x^*$. In particular, $\gna$ with an Armijo linesearch was shown to be effective in \cite{DaPo23}. 

Our choice of $r_{k+1}$ in Algorithm \ref{alg:adapgsg} is partially motivated by Criterion 3.1.\ref{crit3}.
That is, in the case of a nonsingular 
problem, we prefer to use Newton asymptotically rather 
than NA. Hence we want $r_{k+1}$ to tend to 
zero as our solver converges, thereby scaling the 
$\gna$ iterates heavily towards pure Newton steps in the 
asymptotic regime and enjoying quadratic 
convergence locally. The remainder of 
this section is dedicated to quantifying how close 
a $\gna$ iterate 
is to a standard Newton 
iterate in the asymptotic regime when $f'(x^*)$ 
is nonsingular. 
The main result is 
Theorem \ref{thm:nonsing-case}
below which bounds $\|x_{k+1}^{\NA}-x_{k+1}^{\newt}\|$ locally, where 
$x_{k+1}^{\newt} := x_k+w_{k+1}$, and $x_{k+1}^{\NA} := 
x_{k+1}^{\newt}-\lambda_{k+1}\gamma_{k+1}(x_{k+1}^{\newt}-x_k^{\newt})$. This notation is introduced to 
emphasize the Newton and Newton-Anderson iterates in the 
comparison. 
We will also define 
$e_k^{\newt} := x_k^{\newt} - x^* = x_{k-1}+w_k - x^*$. 
The following lemma will be used in the proof of 
Theorem \ref{thm:nonsing-case}. 
Lemma \ref{lem:gammabound} 
bounds $|\lambda_{k+1}\gamma_{k+1}|$, the scaled
$\gamma_{k+1}$ returned by $\gna$ at iteration $k$, in terms of 
$\eta_{k+1}$ and $r_{k+1} = \min\{\eta_{k+1},\hat{r}\}$. 
The proof consists of walking through 
the cases in Algorithm \ref{alg:adapgsg}, and is therefore 
left to the interested reader.

\begin{lemma}\label{lem:gammabound}
Let $\eta_{k+1} = \|w_{k+1}\|/\|w_k\|$ and $\hat{r}\in (0,1)$.
Define 
$r_{k+1} := \min\{\eta_{k+1},\hat{r}\}$ and
$\beta_{k+1} := r_{k+1}\eta_{k+1}$ as in 
Algorithm \ref{alg:adapgsg}.
Let $\lambda_{k+1}$ be the value computed by 
Algorithm \ref{alg:adapgsg} at iteration $k$.  If $\eta_{k+1} < 1$, then $\lambda_{k+1}\gamma_{k+1}$ returned by 
Algorithm \ref{alg:adapgsg} satisfies
\begin{align}\label{eq:gammabound1}
|\lambda_{k+1}\gamma_{k+1}| \leq \frac{\beta_{k+1}}
{1-\beta_{k+1}}.
\end{align}
when $\lambda_{k+1} = 1$, and 
\begin{align}\label{eq:gammabound2}
|\lambda_{k+1}\gamma_{k+1}| = \frac{\beta_{k+1}}
{1+\sign(\gamma_{k+1})\beta_{k+1}}.
\end{align}
when $\lambda_{k+1} < 1$.
\end{lemma}


With Lemma \ref{lem:gammabound}, we can bound 
$\|x_{k+1}^{\NA}-x_{k+1}^{\newt}\|$ in terms of $\|e_k\|$, 
$\|e_{k-1}\|$, and $\eta_{k+1}$. 

\begin{theorem}\label{thm:nonsing-case}
If $f'(x^*)$ is nonsingular, then there exists a 
$\rho>0$ and a constant $C$ depending only on $f$ 
such that for $x_k$ and
$x_{k-1}$ in $B_{\rho}(x^*)$ and $\eta_{k+1} < 1$, 

\begin{align}
\|x_{k+1}^{\NA}-x_{k+1}^{\newt}\| \leq C\left(\frac{\beta_{k+1}}{1 - \beta_{k+1}}\right)\max\{\|e_k\|^2,\|e_{k-1}\|^2\}
\end{align}
when $\lambda_{k+1}=1$, and 
\begin{align}
\|x_{k+1}^{\NA}-x_{k+1}^{\newt}\| \leq C\left(\frac{\beta_{k+1}}{1 +\sign(\gamma_{k+1})\beta_{k+1}}\right)\max\{\|e_k\|^2,\|e_{k-1}\|^2\}
\end{align}
when $\lambda_{k+1} < 1$. 
\end{theorem}

\begin{proof}
Using our notation from the discussion preceding 
Lemma \ref{lem:gammabound}, an iterate generated by 
$\gna$ takes the form $x_{k+1}^{\NA} = x_{k+1}^{\newt} -
\lambda_{k+1}\gamma_{k+1}\left(x_{k+1}^{\newt}-
x_k^{\newt}\right)$. Therefore 
$\|x_{k+1}^{\NA} - x_{k+1}^{\newt}\| = |\lambda_{k+1}
\gamma_{k+1}|\,\|x_{k+1}^{\newt}-x_k^{\newt}\|$.
Since $f'(x^*)$ is nonsingular, we can take $\rho$ 
sufficiently small to ensure that $\|e_{j+1}^{\newt}\|\leq 
C\|e_j\|^2$, where $C$ is a constant determined by 
$f$, when $\|e_j\| < \rho$ for $j=k,k-1$. Hence, upon 
adding and subtracting $x^*$ we obtain 
$\|x_{k+1}^{\newt}-x_{k}^{\newt}\|=\|e_{k+1}^{\newt}
-e_{k}^{\newt}\|\leq 2C\max\{\|e_k\|^2,\|e_{k-1}\|^2\}$.
To complete the proof, we write $2C = C$ 
and apply Lemma \ref{lem:gammabound} to bound 
$|\lambda_{k+1}\gamma_{k+1}|$ for the cases 
$\lambda_{k+1}=1$ and $\lambda_{k+1} < 1$.
\end{proof}

We conclude this section with Corollary \ref{cor:localquad}, 
which proves that $\gna$ can recover local quadratic convergence 
from NA when applied to nonsingular problems.

\begin{corollary}\label{cor:localquad}
    If $f'(x^*)$ is nonsingular, $\|e_{k}\| < \|e_{k-1}\|$, and $\eta_{k+1} < \hat{r}$, then there exists a $\rho>0$ and constants $C_1$ and $C_2$ depending only on $f$ such that 
    \begin{align}
        \|e_{k+1}\|\leq \left(\frac{C_1}{1-\hat{r}^2}+C_2\right)\|e_k\|^2
    \end{align}
    for $x_k,x_{k-1}\in B_{\rho}(x^*)$.
\end{corollary}

\begin{proof}
    Adding and subtracting $e_{k+1}^{\newt}$ to 
    $e_{k+1}$ gives 
        $\|e_{k+1}\|\leq \|x_{k+1}^{\NA} - x_{k+1}^{\newt}\| + \|e_{k+1}^{\newt}\|.$
    By Theorem \ref{thm:nonsing-case}, $\|x_{k+1}^{\NA}-x_{k+1}^{\newt}\|\leq C\beta_{k+1}(1-\beta_{k+1})^{-1}\max\left\{\|e_k\|^2,\|e_{k-1}\|^2\right\}$, and
    for $x_k\in B_{\rho}(x^*)$, $\|e_{k+1}^{\newt}\|\leq C_2\|e_k\|^2$. Since $\eta_{k+1}<\hat{r}$, we have $\beta_{k+1} = \eta_{k+1}^2$. Moreover, when $f'(x^*)$ is nonsingular, Taylor expansion shows that
    \begin{align}
        \eta_{k+1} \leq C_1 \frac{\|e_k\|}{\|e_{k-1}\|}
    \end{align}
    for $x_k,x_{k-1}\in B_{\rho}(x^*)$. Thus $\beta_{k+1}\max\left\{\|e_k\|^2,\|e_{k-1}\|^2\right\} \leq C_1 \|e_k\|^2$, and therefore
    \begin{align}
        \|e_{k+1}\|\leq \left(\frac{C_1}{1-\hat{r}^2}+C_2\right)\|e_k\|^2.
    \end{align}
\end{proof}

\section{Numerics}
\label{sec:numerics}

In this section, we demonstrate the effectiveness of NA and $\gna$ near bifurcation points by applying 
these algorithms to the following parameter-dependent PDEs. All 
computations are performed on an M1 MacBook with GNU Octave 8.2.0.

\subsection{Test Problems}

\begin{enumerate}

\item Navier-Stokes Flow in a Channel
\begin{equation}\label{coanda-model}
{
\small
\left\{
\begin{split}
-\mu \Delta \bu +\bu \cdot\nabla \bu +\nabla p &= \mathbf{0}\\
 \nabla \cdot \bu &= \mathbf{0}
\end{split}
\qquad\qquad
\begin{split}
\mathbf{u} &= \mathbf{u}_{\text{in}},\hspace{0.5em} \Gamma_{\text{in}},\\
\mathbf{u} &= 0,\hspace{0.5em} \Gamma_{\text{wall}},\\
-p\mathbf{n} + (\mu\nabla \mathbf{u})\mathbf{n} &= 0,\hspace{0.5em} \Gamma_{\text{out}}.
\end{split}
\right.
}
\end{equation}

\item {Rayleigh-B\'enard Convection}

\begin{equation}\label{rayben-model}
{
\small
\left\{
\begin{split}
-\mu\Delta\bu  + \bu \cdot \nabla\bu +\nabla p - \text{Ri}\, T \mathbf{e}_y &= \mathbf{0}\\
\nabla \cdot \bu &= 0 \\
-\kappa\Delta T +\bu\cdot\nabla T &= 0\\
\end{split}
\qquad\qquad
\begin{split}
T &= 1, \hspace{0.25em} \Gamma_1 := \{1\} \times (0,1),\\
T &= 0, \hspace{0.25em} \Gamma_2 := \{0\} \times (0,1),\\
\nabla T\cdot \bn &= 0, \hspace{0.25em} \Gamma_3 := (0,1) \times \{0,1\},\\
\bu &= \mathbf{0}, \hspace{0.25em} \partial\Omega = \Gamma_1\cup\Gamma_2\cup\Gamma_3.
\end{split}
\right.
}
\end{equation}

\end{enumerate}

In both models, $\mathbf{u}$ denotes the 
fluid velocity, $p$ the pressure, and $\mathbf{n}$ 
the outward normal. In Model \eqref{coanda-model}, $\mu$ denotes 
the viscosity parameter. In Model \eqref{rayben-model}, $T$ denotes
the temperature of the fluid, and $\ri$ denotes the 
Richardson number, which is the parameter of interest 
for Model \eqref{rayben-model}. We set $\mu=\kappa = 10^{-2}$. 
These parameters values are chosen 
so as to replicate the results seen in \cite{GaLiReWi12} as $\ri$ ranges from $3.0$ to $3.5$.
For the flow in a channel, Model \eqref{coanda-model}, we use $\mathcal{P}_2-\mathcal{P}_1$ Taylor-Hood elements \cite[p.~164]{braess07}. The channel, 
shown in Figure \ref{fig:meshes}, is arranged such that the left most boundary lies at $x=0$, the right most boundary lies 
at $x=50$, and 
the boundary components are given by 
$\Gamma_{\text{in}} = \{0\} \times [2.5,5]$, $\Gamma_{\text{out}} = \{50\} \times [0,7.5]$, and 
$\Gamma_{\text{wall}} = [0,10]\times (\{2.5\}\cup\{5\}) \cup \{10\} 
\times ([0,2.5]\cup [5,7.5]) \cup [10,40]\times (\{0\}\cup \{7.5\})$.
For the Rayleigh-B\'endard Model, we use $\mathcal{P}_2-\mathcal{P}_1^{\text{disc}}$ 
Scott-Vogelius elements \cite{GuSc19}, where $\mathcal{P}_1^{\text{disc}}$ denotes piece-wise 
linear discontinuous elements. 
The $\mathcal{P}_2-\mathcal{P}_1^{\text{disc}}$ elements are stable 
on Alfeld-split, also known as barycenter-split, triangulations
\cite[p.~77]{qin-thesis}.
The meshes used for 
each model are shown below.

\begin{figure}[H]
\centering
\includegraphics[width=70mm,height=60mm]{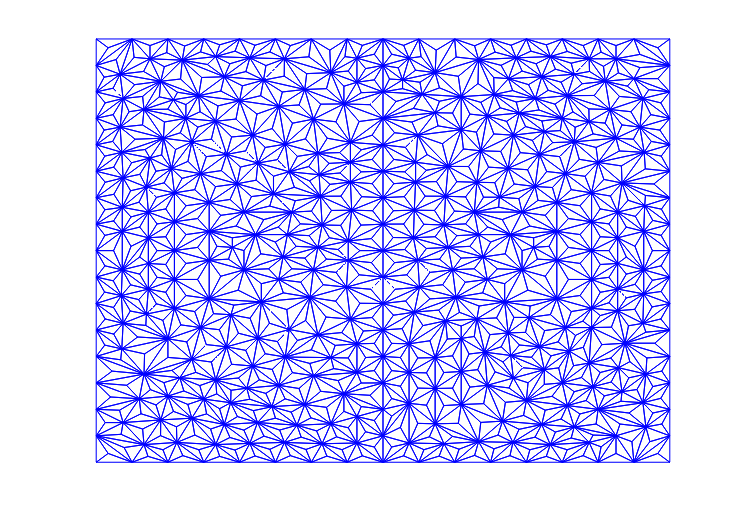}
\includegraphics[width=120mm,height=25mm]{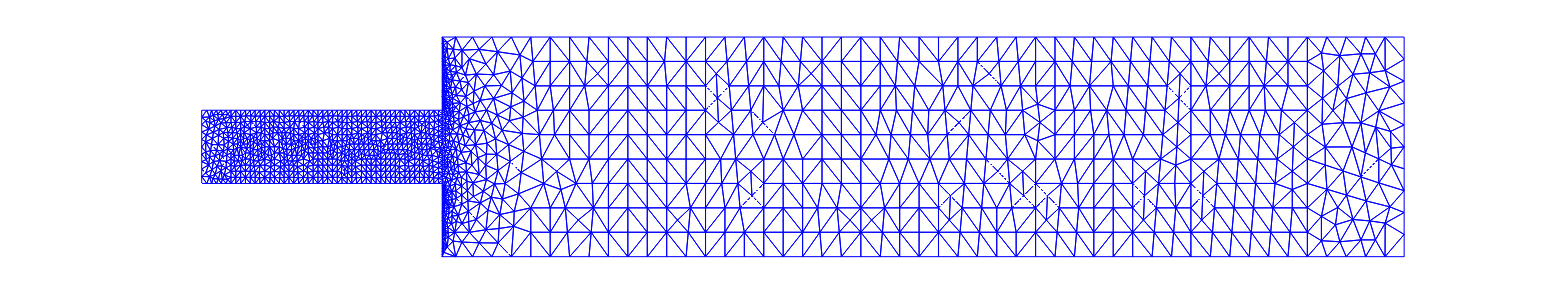}
\hspace{0.25mm}
\caption{ \label{fig:meshes} Meshes used for benchmark problems. Top: Mesh used for {Rayleigh-B\'enard}
model. Bottom: mesh used for flow in a channel.}
\end{figure}

For Model \eqref{coanda-model}, it is known \cite{pichithesis} that there exists a critical 
viscosity $\mu^*\in (0.9,1)$ at which a bifurcation occurs. For $\mu > \mu^*$, 
the stable velocity solution is symmetric about 
the center horizontal ($y=3.75$) as seen in the top plot 
of Figure \ref{fig:sym-asym-vels}. For $\mu < \mu^*$, 
there is still a symmetric solution, but it is unstable. Stability is inherited by 
two asymmetric solutions seen in the bottom two plots of 
Figure \ref{fig:sym-asym-vels}. 

\begin{figure}[H]
\centering
\includegraphics[width=90mm]{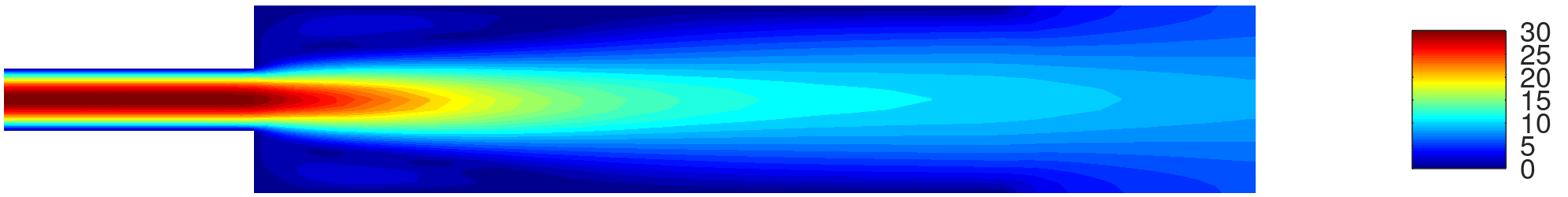}
\includegraphics[width=90mm]{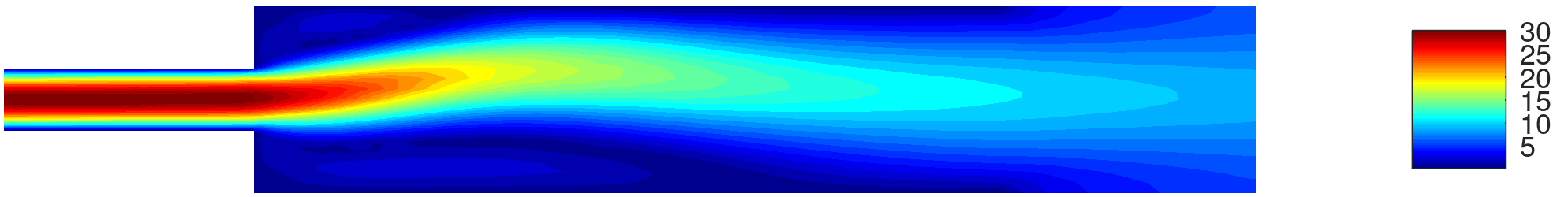}
\includegraphics[width=90mm]{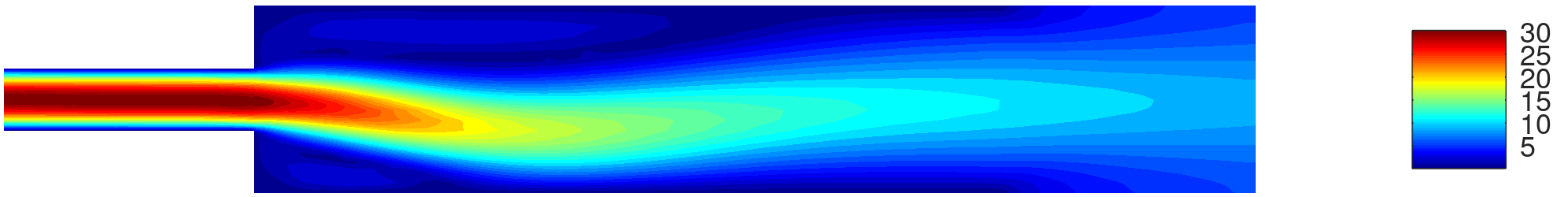}
\caption{ \label{fig:sym-asym-vels} Solutions to channel flow 
problem \eqref{coanda-model} for different $\mu$. Top: Representative
 symmetric solution for $\mu = 1.0$. Middle: Representative
 asymmetric solution with positive vertical velocity 
upon exiting the narrow channel for $\mu = 0.9$. 
Bottom: Representative 
asymmetric solution with negative vertical velocity 
upon exiting the narrow channel for $\mu = 0.9$.}
\end{figure}

The parameter region of interest for Model \eqref{rayben-model} is $\ri\in[3,3.5]$. In this range, 
the flow appears to be in transition from a single eddy 
in the center of the domain to 
two eddies as seen in Figure \ref{fig:eddies} below.

\begin{figure}[H]
\centering
\includegraphics[width=55mm,height=50mm]{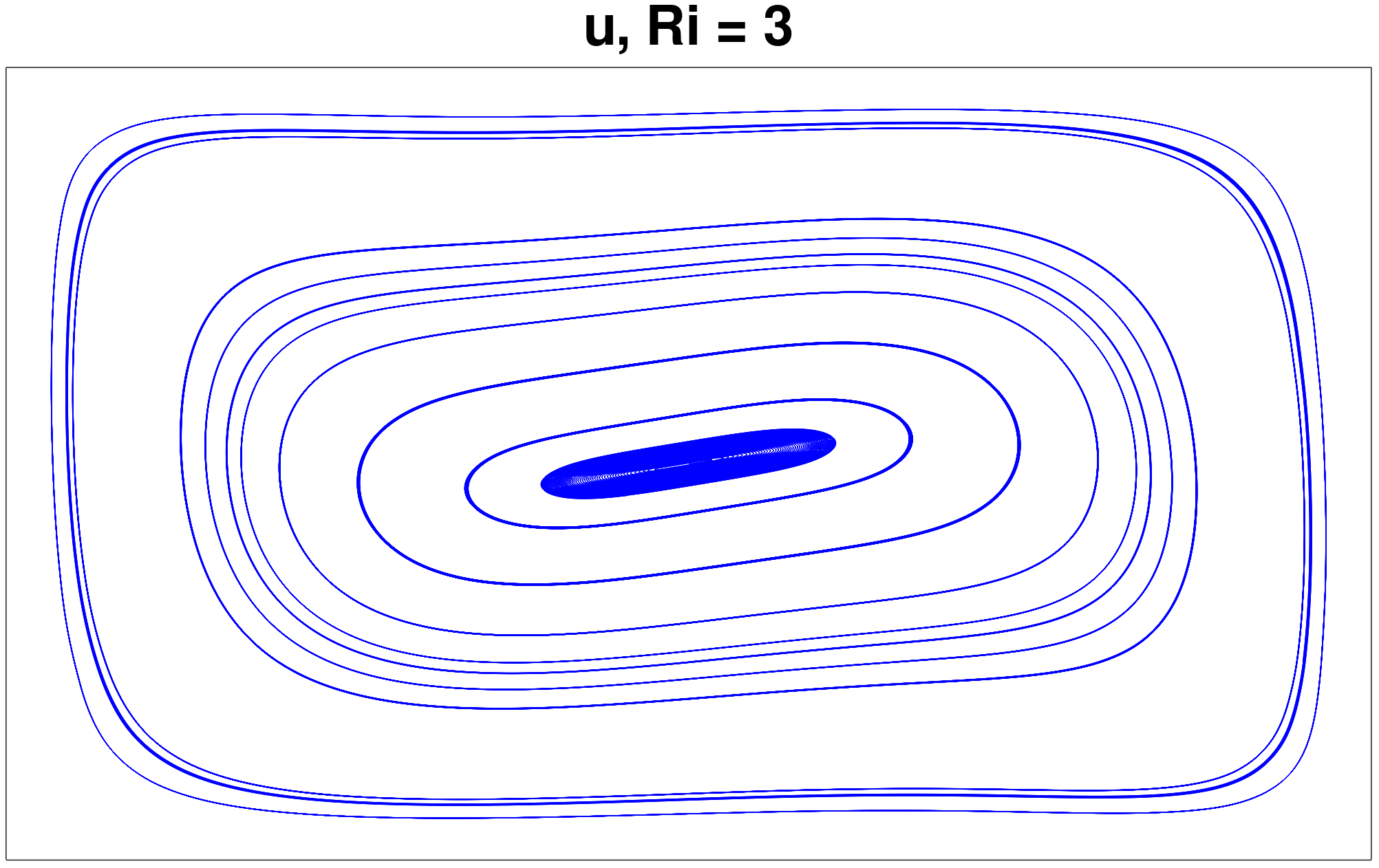}
\includegraphics[width=55mm,height=50mm]{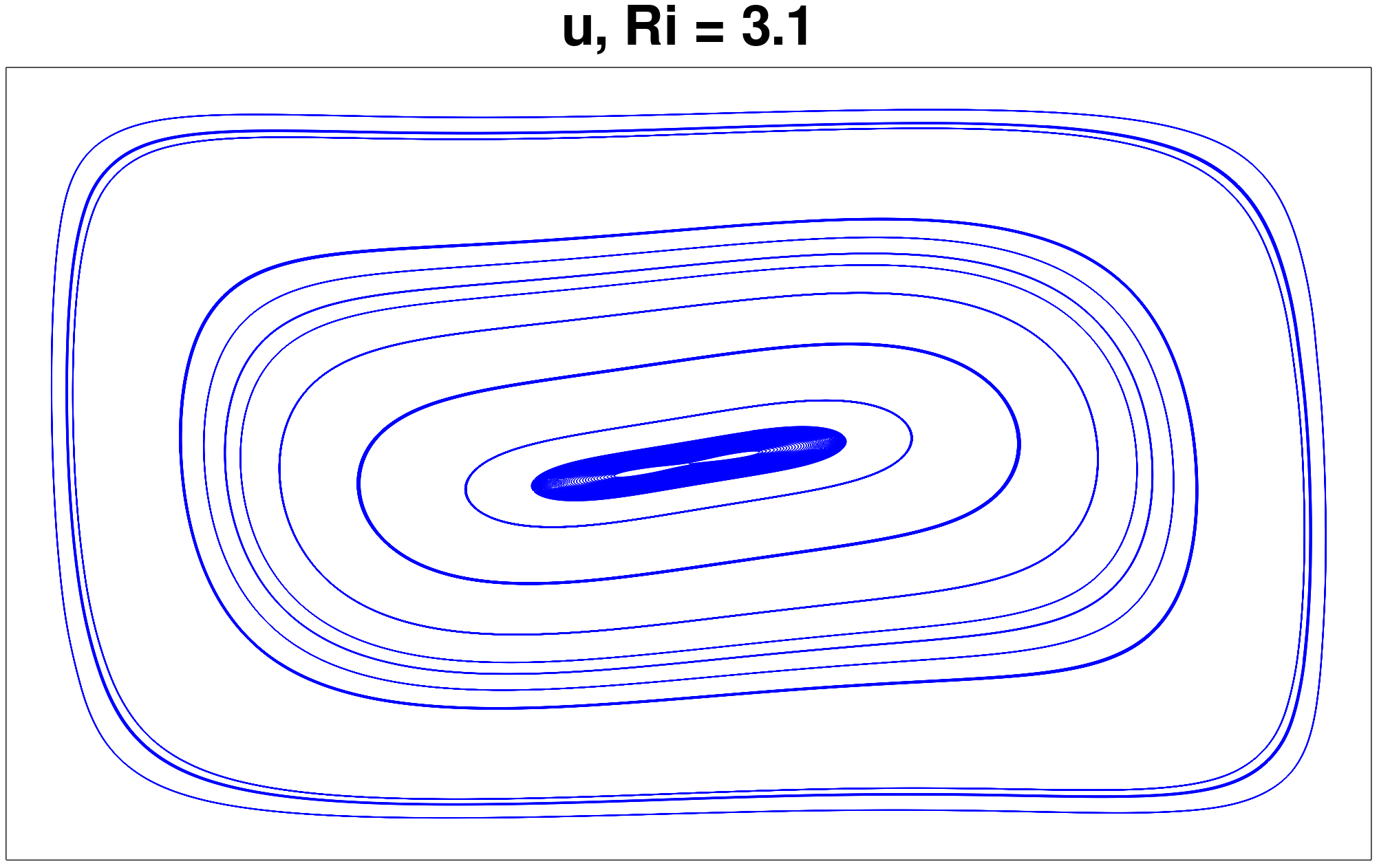}
\includegraphics[width=55mm,height=50mm]{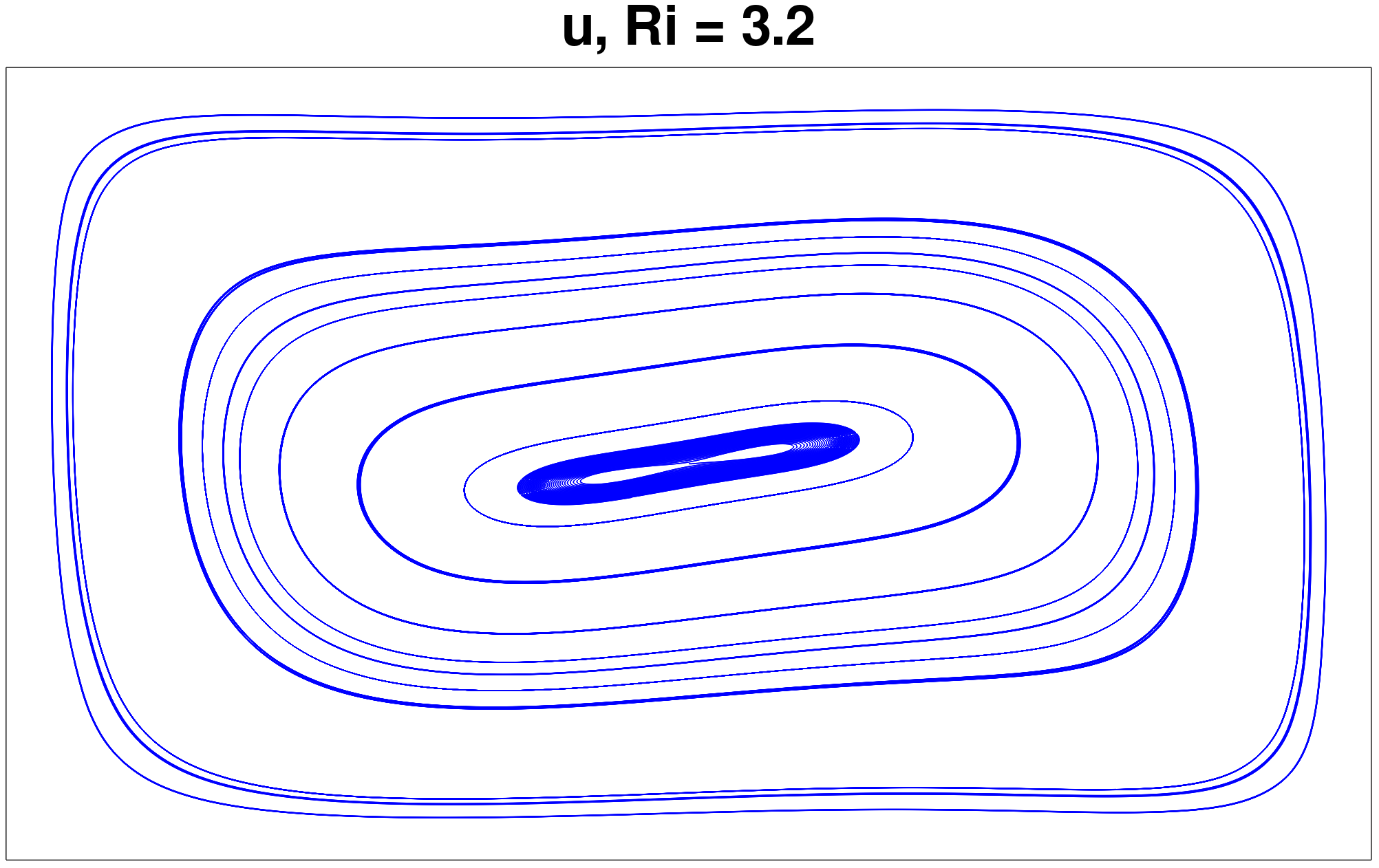}
\includegraphics[width=55mm,height=50mm]{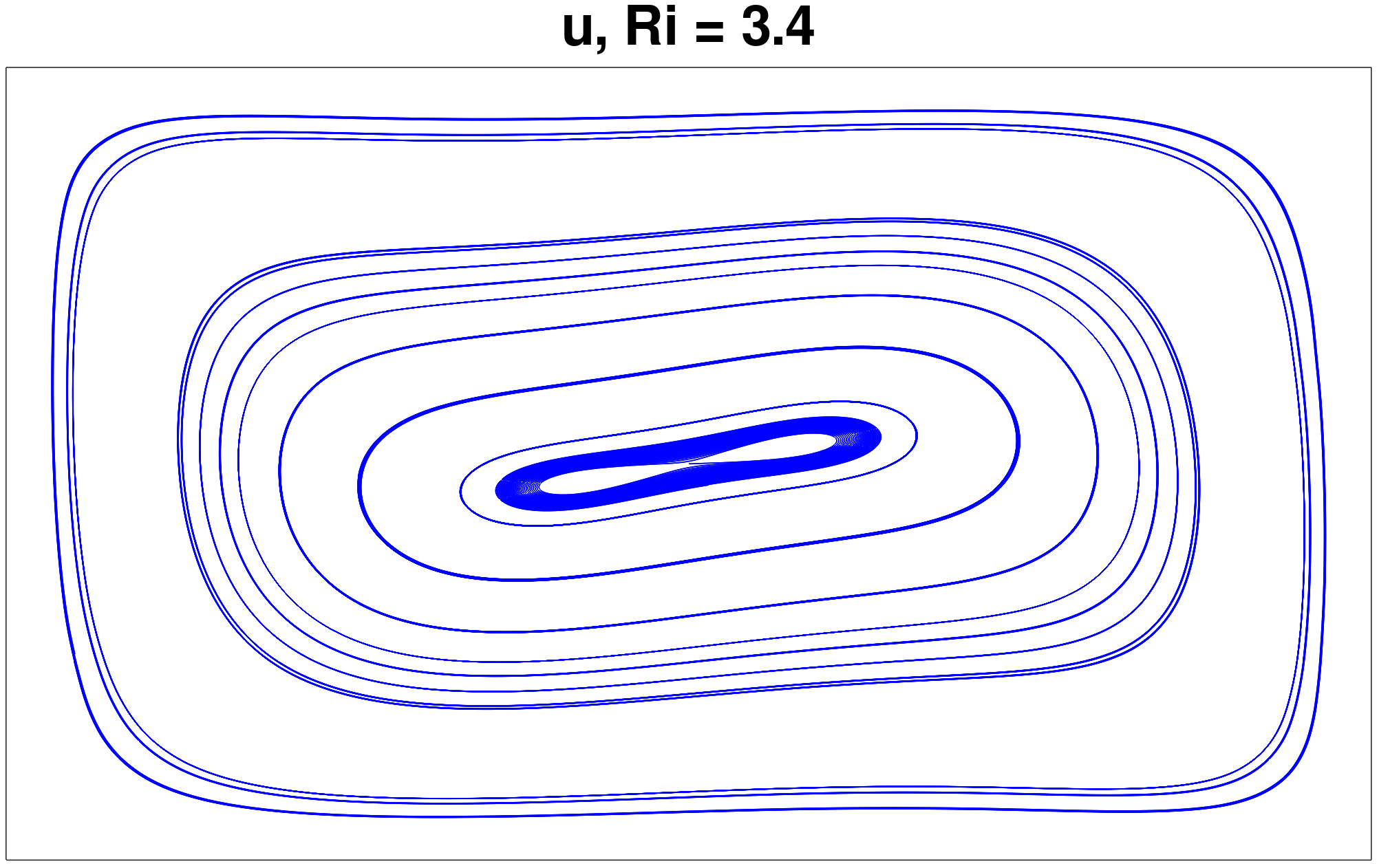}
\caption{ \label{fig:eddies} Velocity streamlines for velocity $u$
from Model \eqref{rayben-model} showing transition from 
one eddy to two eddies. Top Left: Ri = 3.0. Top Right: Ri = 3.1.
Bottom Left: Ri = 3.2. Bottom Right: Ri = 3.4.}
\end{figure}

For Model \eqref{coanda-model}, we take the zero vector 
as our initial iterate. For Model \eqref{rayben-model}, we initialize 
our iterate by applying a Picard step and then applying 
$\gna$. The $H^1$ seminorm is used in our implementations \cite{braess07}. 
Before we discuss the numerical results, we recall that  
\begin{itemize}
\item Newt is Algorithm \ref{alg:newt}. 
\item NA is Algorithm \ref{alg:na1}.
\item NA($m$) is Algorithm \ref{alg:nam}.
\item $\gamma$NA$(r)$, Algorithm \ref{alg:gsgna}, is Algorithm \ref{alg:na1} 
with line 5 replaced by Algorithm \ref{alg:gsg-og}. 
\item $\gna$, Algorithm \ref{alg:adapgsgna}, is Algorithm \ref{alg:na1} with line 
5 replaced by Algorithm \ref{alg:adapgsg}.  
\end{itemize}
For $\gamma$NA$(r)$ and $\gna$, we set $r$ and $\hat{r}$ 
respectively to a fixed quantity for all iterations. 

\subsection{General Discussion of Results} 
\label{sec:gendiscussion}

The following experiments demonstrate three strategies for solving nonlinear problems
near bifurcations using NA and $\gna$. The first two are 
{\it asymptotic safeguarding} and {\it preasymptotic 
safeguarding}. With asymptotic safeguarding, 
we run NA until
the residual is smaller
than some user-chosen threshold, and 
we use $\gna$ for all subsequent iterates. Hence the 
solve will behave like NA until the last few iterations when 
$\gna$ is applied. For the 
experiments in Section \ref{sec:asymsg}, we activated 
$\gna$
when $\|w_{k+1}\| < 10^{-1}$. 
We chose $10^{-1}$ since 
we want to activate $\gamma$-safeguarding 
as early as possible. 
This allows for earlier detection of a nonsingular problem, and thus 
faster convergence. We found that 
activating $\gamma$-safeguarding for $\|w_{k+1}\|<\tau<10^{-1}$ will not 
necessarily break convergence, 
but if the problem is nonsingular, this will
not be detected as early.
Activating $\gna$ 
when $\|w_{k+1}\| < 1$ can break convergence, 
though this
seems to be problem-dependent.
Convergence for Model \eqref{coanda-model} was virtually unaffected 
with threshold $\|w_{k+1}\| < 1$, but for Model \eqref{rayben-model}
 with $\ri = 3.5$,
setting the activation threshold to $1$ caused 
$\gnaarg{0.9}$ 
to diverge when it had converged with threshold $0.1$.
With activation threshold $0.1$, asymptotic safeguarding 
is shown to be effective 
close to the bifurcation point, where Newton's method 
can fail to converge. 
Preasymptotic safeguarding, on the other hand, applies 
$\gna$ at each step of the solve. Interestingly, we observe 
that $\gna$ applied preasymptotically can outperform NA when applied to Model \eqref{rayben-model}, and even recover 
convergence when both Newton and $\NA$ diverge (see 
Figure \ref{fig:preasym-ri3.2}). A theoretical explanation for 
this requires a better understanding of these methods in the 
preasymptotic regime. In particular, a better understanding 
of the descent properties of Anderson acceleration would be 
of great value. It is known \cite{TuWaSh02} that for 
singular problems in the 
preasymptotic regime, where $\|f(x)\|$ is not small, the 
Newton update step $s_k = -f'(x_k)^{-1}f(x_k)$, 
or $w_{k+1}$ in our notation, can be large 
and nearly orthogonal to the gradient of $\|f\|_2^2$. 
With $\NA$, our update step takes the form 
$s_k^{\NA}=w_{k+1}-\gamma_{k+1}(x_{k+1}^{\newt}-x_k^{\newt})$. 
It is clear that $s_k^{\NA}$ is a descent direction for 
sufficiently small $\gamma_{k+1}$, since in this case it is
nearly $w_{k+1}$. It is possible that for certain values of 
$\gamma_{k+1}$, $s_k^{\NA}$ is a stronger descent direction 
than $s_k$. 

Another possible explanation as to why $\gna$ can outperform $\NA$ 
in some cases is its resemblance to restarted Anderson 
acceleration methods. Restarted versions of Anderson 
acceleration are often applied in various 
forms for depth $m>1$, and 
have been shown to effective in practice \cite{ChDuLeSe21,HeZhXiHoSa22,HeVa19,Pa19}. In the special case of Newton-Anderson with depth $m=1$, every odd iterate is simply 
a Newton step, rather than 
a combination of the previous two Newton steps. Hence 
the algorithm is ``restarted" every other step. Explicitly,
we have for $k\geq 1$,  
\begin{align*}
x_{2k-1} &= x_{2k-1}^{\newt}\\
x_{2k} &= x_{2k}^{\newt} - \gamma_{2k}(x_{2k}^{\newt}-
x_{2k-1}^{\newt})
\end{align*}
In other words, $\gamma_{2k-1} = 0$ for all $k\geq 1$. With 
$\gna$, $\gamma_{k+1}$ is not necessarily set to zero, but 
it is scaled towards zero, significantly so depending on 
$\beta_{k+1}$. In this way, one may think of $\gna$ as 
a {\it quasi-restarted} Anderson scheme when the depth 
$m = 1$. At the moment, $\gamma$-safeguarding is not 
developed for $m > 1$, but this interpretation of 
$\gna$ as a quasi-restarted method could lead to such 
a development. 
Presently, these are only heuristics, but they provide
interesting questions for future projects.

The third technique we demonstrate to solve these problems 
near bifurcation points is increasing the depth 
$m$. Evidently, the right choice of $m$ can significantly 
improve convergence by reducing the number of iterations
to convergence by half, and increase the domain of 
convergence with respect to the parameter. We found, however, that such 
performance was very sensitive to the choice of $m$. So while
the results suggest this could be developed into a viable 
strategy, more work is required to achieve this.

In all experiments we take $\hat{r}\in (0,1)$ 
since this is the range in which local convergence of $\gna$ 
is guaranteed by Theorem \ref{thm:convergence-chap3}. In 
practice, neither Algorithm \ref{alg:gsgna} nor Algorithm 
\ref{alg:adapgsgna} breaks down if one sets $r$ or $\hat{r}$ 
respectively to zero, one, or a value greater than one. 
Setting $r$ or $\hat{r}$ to zero reduces the iteration to 
Newton, and choosing one leads to a more NA like iteration in 
the preasymptotic regime. A systematic 
study of $\gna$ with $\hat{r} \geq 1$ has not been performed, 
but experiments performed 
thus far show no significant advantage over $\hat{r} \in (0,1)$. 

What the best choice of $\hat{r}$ is remains an open 
question. Numerical experiments suggest the best choice 
depends on the initial guess $x_0$. For example, when applied to the channel flow problem preasymptotically (see Section 
\ref{sec:pasymsg}), setting $\hat{r} = 0.9$ results in faster convergnece than Newton if $x_0$ is the zero vector. 
If we perturb this $x_0$ (discussed in Section \ref{sec:preasym_channel_flow}), then the choice of $\hat{r} = 0.9$ leads to $\gna$ converging slower than Newton, while $\hat{r}=0.5$ outperforms NA and Newton. This phenomena, that the 
best choice of $\hat{r}$ in $\gna$ varies with $x_0$, is 
seen with other choices of $x_0$ as well. Elucidating this 
dependence is the subject of ongoing work.

\subsection{Asymptotic Safeguarding}
\label{sec:asymsg}

Under the assumptions of Theorem \ref{thm:convergence-chap3}
or Theorem \ref{thm:nonsing-case}, $\gna$ is guaranteed to converge locally. This motivates
the strategy of this subsection. 
As discussed above, asymptotic safeguarding is when we only 
apply $\gna$ once the residual is smaller than a set 
threshold.
This allows one to take full advantage of NA in the 
preasymptotic regime, and 
ensures fast quadratic convergence for nonsingular problems. 
In practice, this means fast local convergence provided 
NA reaches the domain of convergence. 
We activate $\gna$ when $\|w_{k+1}\| < 10^{-1}$ 
in the experiments below. When $\|w_{k+1}\| > 10^{-1}$, 
we run NA. 

\subsubsection{Results for Channel Flow Model}

The results when applied 
to Model 
\eqref{coanda-model} are shown below in 
Figures \ref{fig:asym-mu0.96}, \ref{fig:asym-mu0.94}, and \ref{fig:asym-mu0.92}.
The takeaway is that when applied asymptotically, convergence 
of $\gna$ is not as sensitive to the choice of $\hat{r}$ as it is when applied in the preasymptotic regime (see 
Section \ref{sec:pasymsg}), and 
Algorithm \ref{alg:adapgsg}
is working as intended 
by detecting that the problem is 
nonsingular.
This is seen in the plot on the right of 
Figures \ref{fig:asym-mu0.96}, \ref{fig:asym-mu0.94}, and \ref{fig:asym-mu0.92}.
Recall that $r_{k+1}$ is the adaptive parameter in Algorithm \ref{alg:adapgsg} 
that determines how close a 
$\gna$
step is to a Newton step. 
When $r_{k+1}\approx 0$, 
and the criteria in Algorithm \ref{alg:adapgsg} is met, 
the 
$\gna$
step $x_{k+1}$ will be close to a Newton step. When 
$r_{k+1} \approx 1$, this scaling will be much less severe, 
and the $\gna$ step $x_{k+1}$ will 
be close to an NA step. From our discussion 
in Section \ref{sec:adapgsg}, we want $r_{k+1} \to 0$ when the problem is 
nonsingular in order to enjoy local quadratic convergence. This is not guaranteed 
with NA \cite{ReXi23}. Observing the $r_{k+1}$ plots in Figures \ref{fig:asym-mu0.96},
\ref{fig:asym-mu0.94}, and \ref{fig:asym-mu0.92}, one notes that $r_{k+1}\to 0$ as the solver converges. Since $r_{k+1} = \min\{\eta_{k+1},\hat{r}\}$, this is equivalent to $\eta_{k+1}\to 0$ as the solve converges. Thus by Theorem \ref{thm:nonsing-case}, the $\gna$ iterates converge to
Newton iterates asymptotically. This is precisely what $\gna$ was designed to do: detect nonsingular problems, and respond by scaling the iterates towards Newton
asymptotically. Since $\gna$
is only activated when $\|w_{k+1}\| < 10^{-1}$ in these examples of asymptotic safeguarding, $r_{k+1}$ is only computed, and plotted, for the last two iterations. In the next section on preasymptotic safeguarding, a more interesting 
$r_{k+1}$ history is seen. Our methods, including 
$\NA$, failed to converge for $\mu < 0.92$. If one wishes 
to solve a problem at a particular parameter, and a direct 
solve fails like we see here for $\mu < 0.92$, one could still
employ $\NA$ or $\gna$ to solve the problem directly for 
a parameter value close to the desired one to obtain an 
initial guess for continuation. The benefit here is that 
the continuation would be required in a smaller parameter 
range, thus reducing the total number of solves required.
We will see in Section \ref{sec:incnewtdepth} that 
increasing $m$ can lead to 
convergence for a wider range of parameters.

\begin{figure}[H]
\centering
\includegraphics[width=80mm]{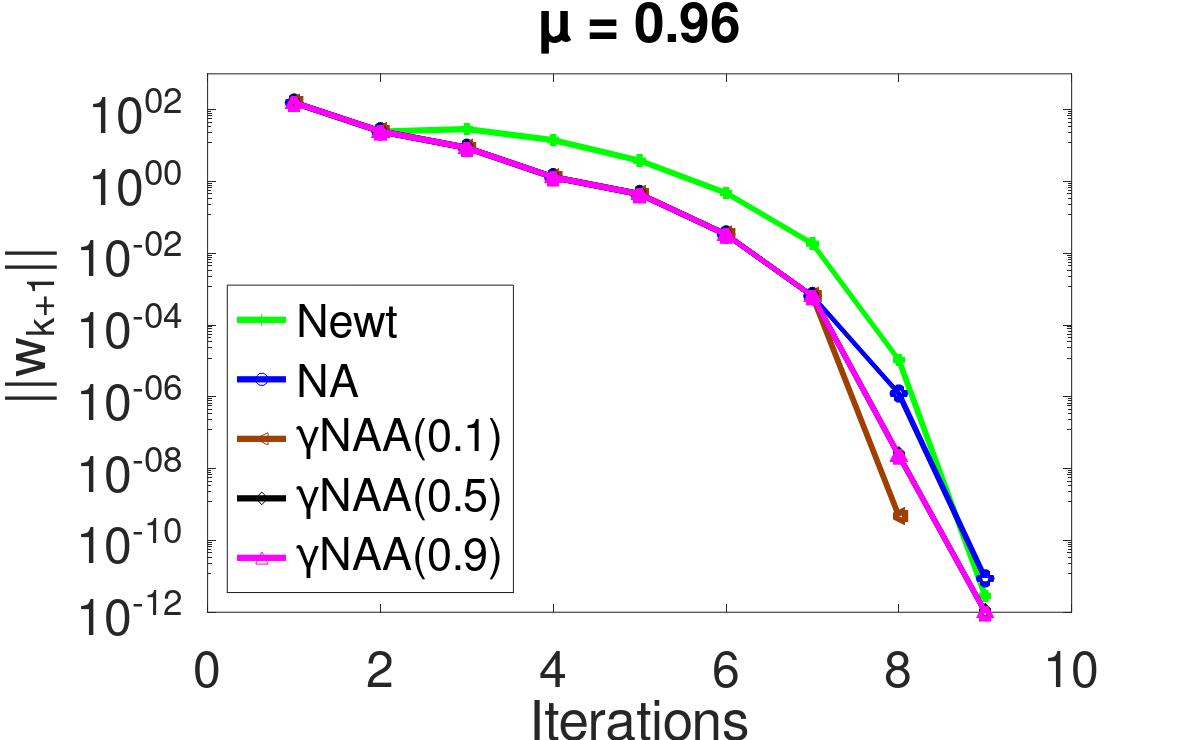}
\includegraphics[width=80mm]{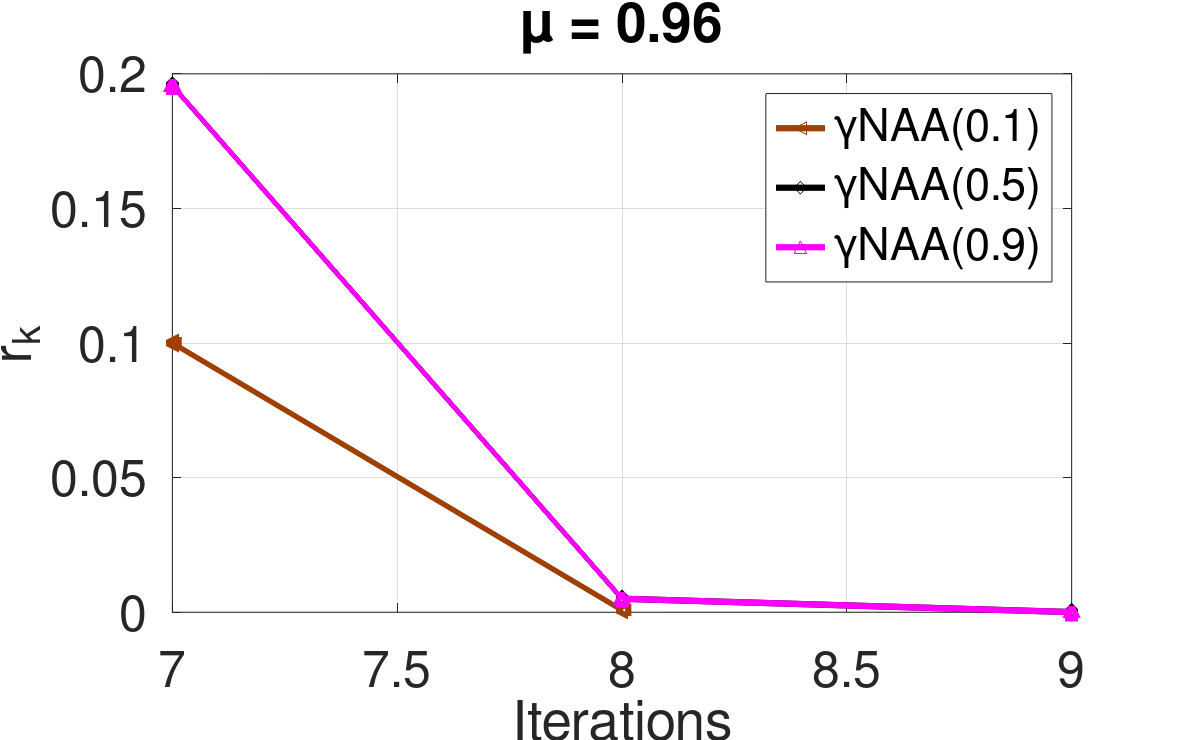}
\caption{ \label{fig:asym-mu0.96} Comparison of 
$\gna$ applied asymptotically with Newton and NA applied 
to Model \eqref{coanda-model} with $\mu = 0.96$.}
\end{figure}

\begin{figure}[H]
\centering
\includegraphics[width=80mm]{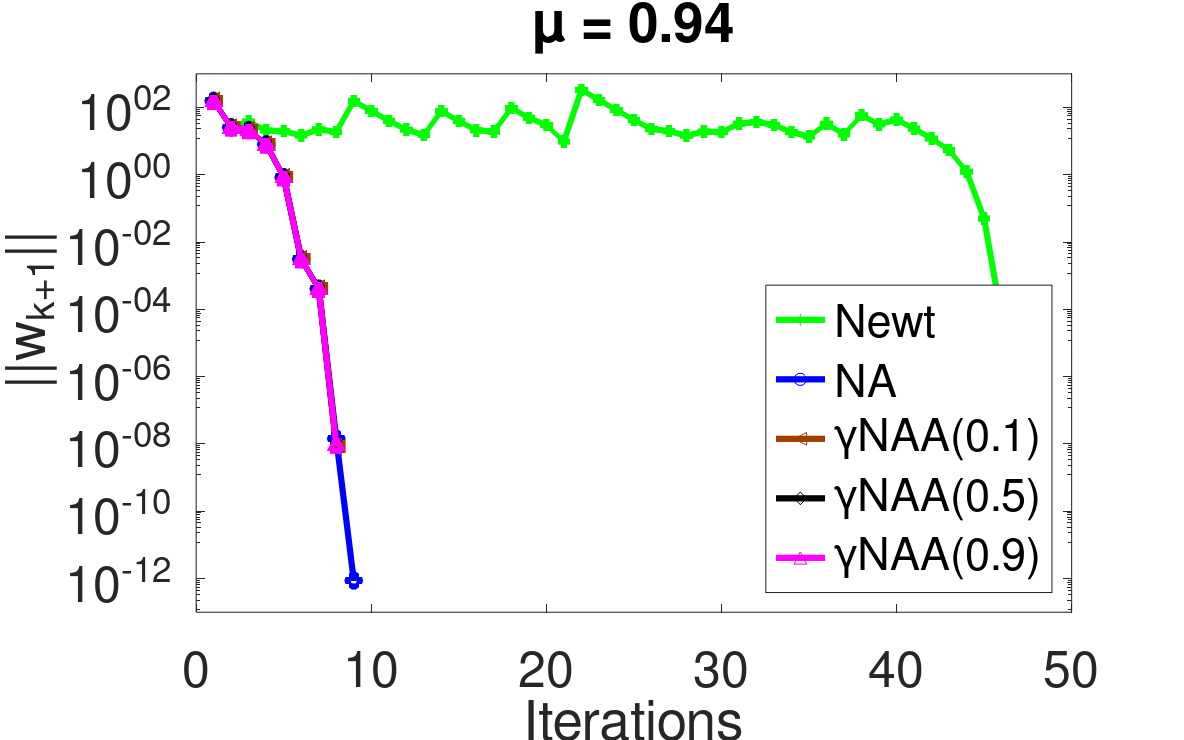}
\includegraphics[width=80mm]{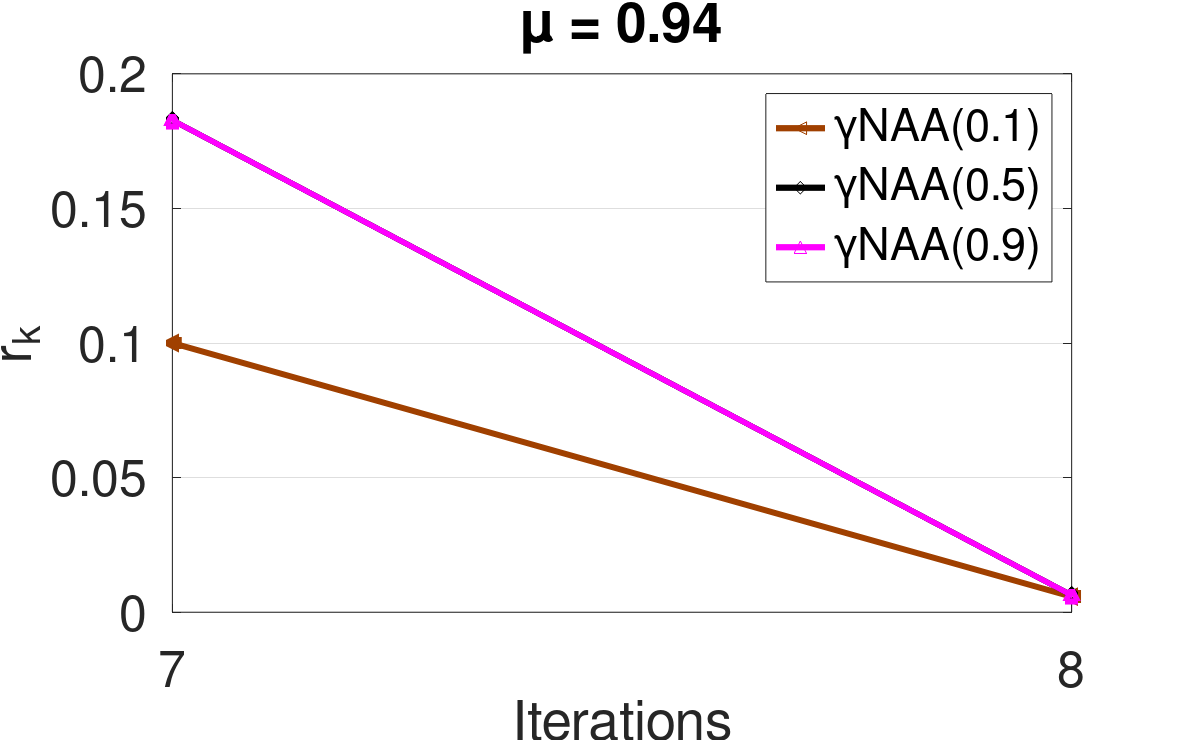}
\caption{ \label{fig:asym-mu0.94} Comparison of 
$\gna$ applied asymptotically with Newton and NA applied 
to Model \eqref{coanda-model} with $\mu = 0.94$.}
\end{figure}

\begin{figure}[H]
\centering
\includegraphics[width=80mm]{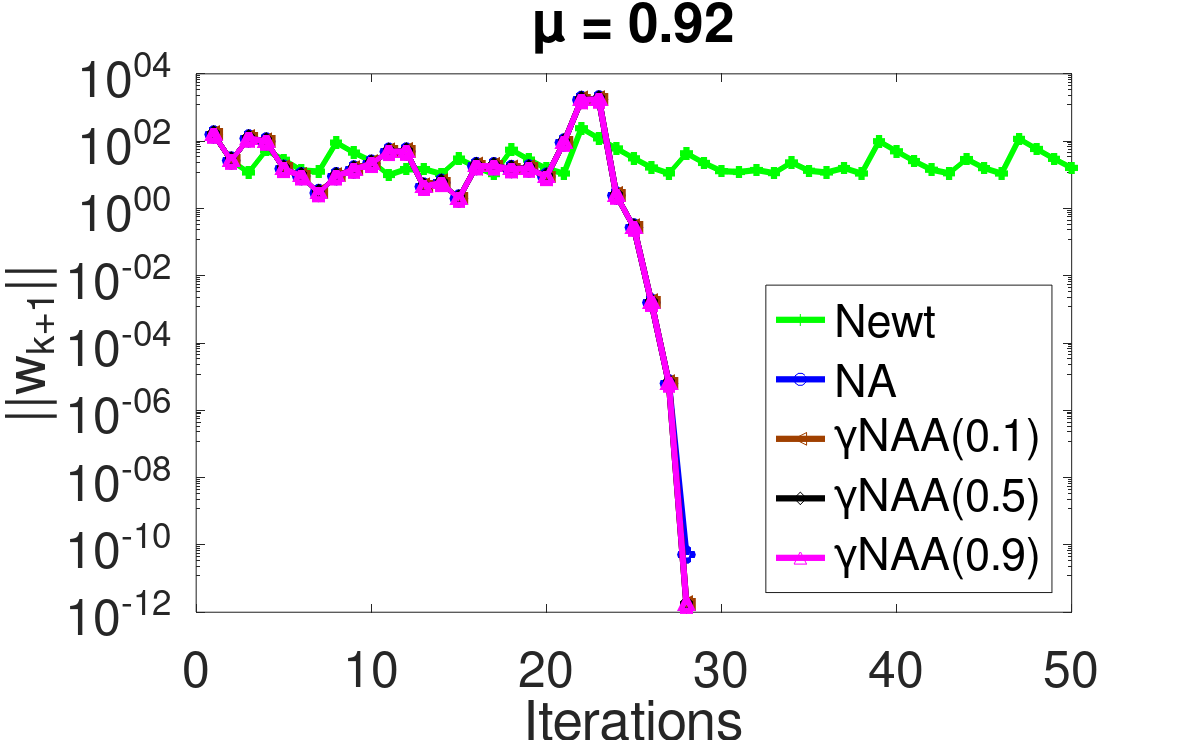}
\includegraphics[width=80mm]{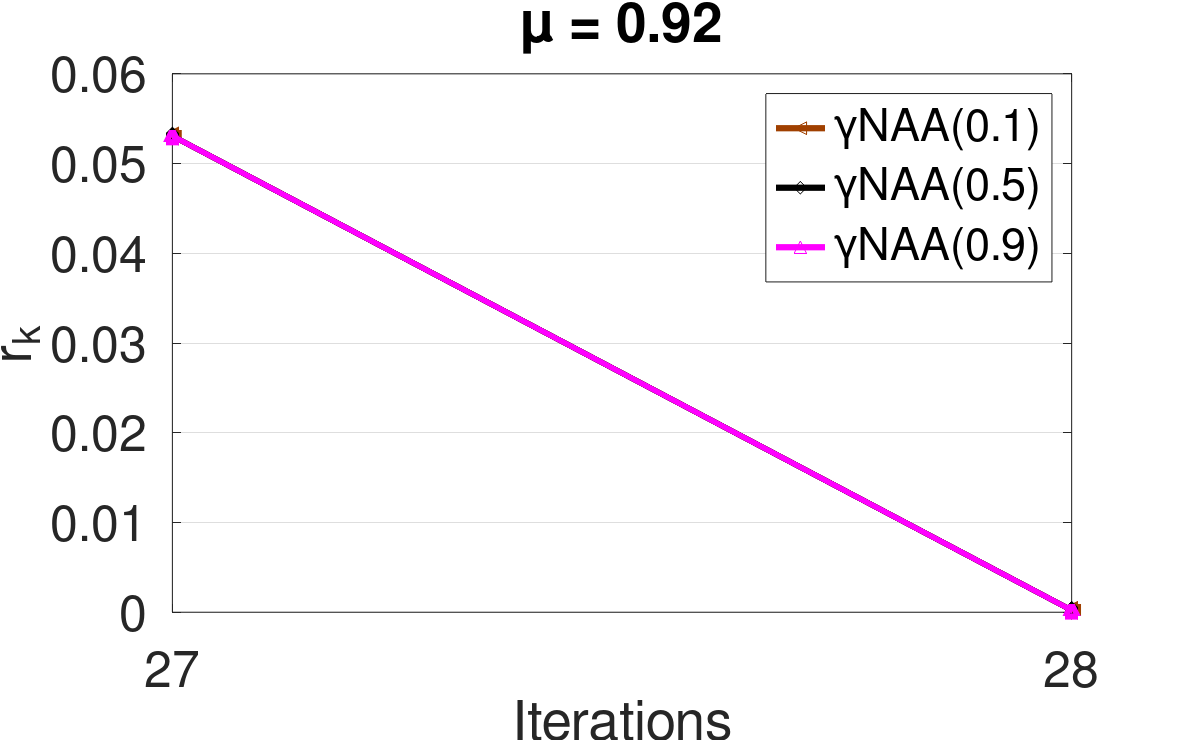}
\caption{ \label{fig:asym-mu0.92} Comparison of 
$\gna$ applied asymptotically with Newton and NA applied 
to \eqref{coanda-model} with $\mu = 0.92$.}

\end{figure}

\subsubsection{Results for Rayleigh-B\'enard Model}
\label{sec:rayben-asymsg}

The results of asymptotic safeguarding with activation
threshold $0.1$ applied to Model \eqref{rayben-model} are similar to 
those of Model \eqref{coanda-model} seen in the previous section. 
In this case, however, NA diverged for $\ri = 3.0$ and 
$\ri = 3.2$. For $\ri = 3.0$, Newton's method converges, but 
all methods diverged for $\ri = 3.2$. In Section 
\ref{sec:pasymsg}, we are able to recover convergence 
for $\ri = 3.2$ with preasymptotic safeguarding. 
For $\ri = 3.1$, 
$\ri = 3.3$, $\ri = 3.4$, and $\ri = 3.5$, Newton diverged
 while NA and $\gna$ converged. Like with Model \eqref{coanda-model}, we see $r_{k+1}\to 0$ as the solve converges. The results for 
$\ri = 3.4$, shown below, are representative of the others for which NA converged. 

\begin{figure}[H]
\centering
\includegraphics[width=80mm]{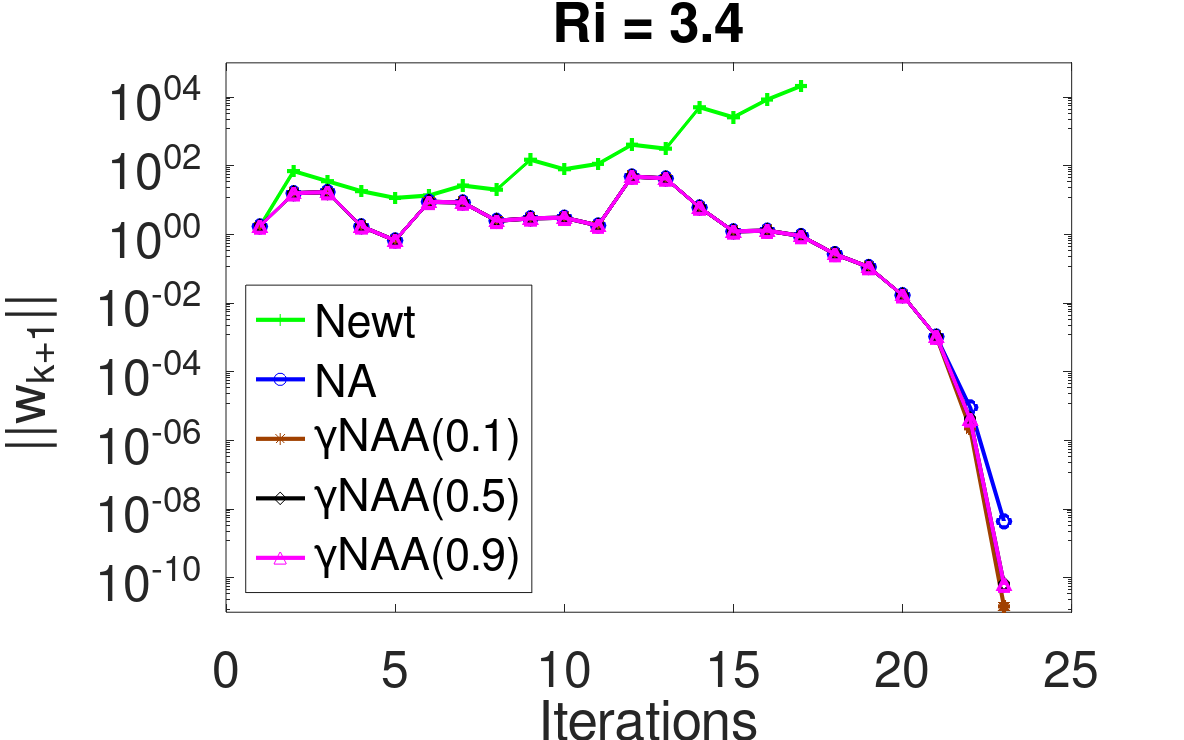}
\includegraphics[width=80mm]{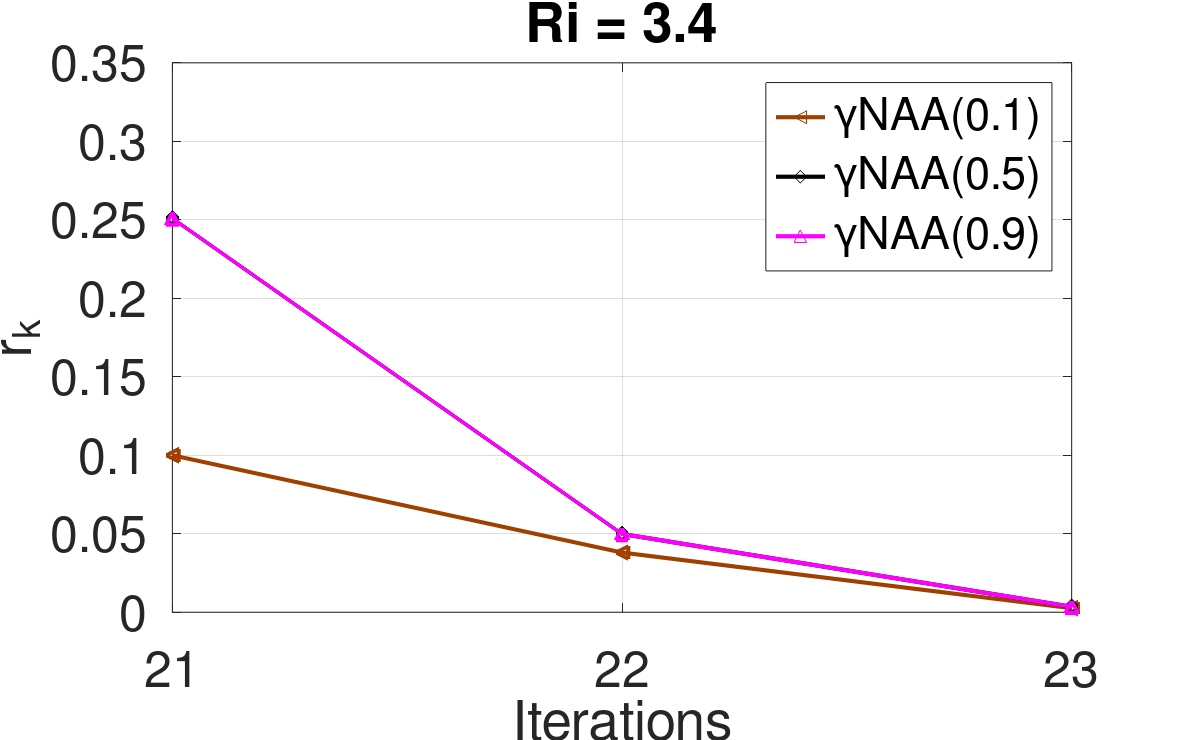}
\caption{ \label{fig:asym-ri3.4} Comparison of 
$\gna$ applied asymptotically with Newton and NA applied 
to Model \eqref{rayben-model} with $\ri = 3.4$.}
\end{figure}

\subsection{Preasymptotic Safeguarding}
\label{sec:pasymsg}

In this section, we demonstrate the preasymptotic safeguarding strategy, where $\gna$ is activated starting at iteration $k=2$, the first iteration where $\NA$ can be applied. Compared to asymptotic safeguarding, preasymptotic safeguarding is more sensitive to the choice of $\hat{r}$, but, with the right choice of $\hat{r}$, it can recover convergence when both Newton and $\NA$ fail. 

\subsubsection{Results for Channel Flow Model}
\label{sec:preasym_channel_flow}
Here we present the results from applying the preasymptotic 
safeguarding strategy to the channel flow Model 
\eqref{coanda-model}. Evidently, Newton's method can still converge quickly with the right initial guess when 
$\mu=0.96$. Figure \ref{fig:mu0.96} demonstrates that in this case, $\NA$ and $\gna$ also 
perform well. The right-most plot in Figure \ref{fig:mu0.96} demonstrates that the $\gna$ is 
working as intended. That is, $r_{k+1}\to 0$, and therefore 
$\gna$ is detecting that Newton is converging quickly,
and responds by scaling its update steps
towards a pure Newton iteration. 
With $\mu = 0.94$, we are closer to the bifurcation point, and observing the left plot in Figure 
\ref{fig:mu0.94}, 
we see that 
Newton takes many more iterations to converge. However, after a long preasymptotic 
phase, Newton does eventually converge quickly, which suggests that the problem is not
truly singular for $\mu = 0.94$. This is again detected by $\gna$. 
Even though the $\gna$ algorithms take a few more iterations to converge than NA with 
$\mu = 0.94$, 
the terminal order of convergence of $\gna$ is greater than that of NA. Approximating 
the rate by $q_{k+1} = \log(\|w_{k+1}\|)/\log(\|w_k\|)$ at each step $k$, and letting
$q_{\text{term}}$ denote the terminal order, we found that 
$q_{\text{term}} = 1.537$ for NA, $q_{\text{term}} = 3.386$ for $\gnaarg{0.1}$, and 
$q_{\text{term}} = 2.091$ for $\gnaarg{0.9}$. 
In the right most plot in Figure 
\ref{fig:mu0.94}, we take as our initial iterate the 
zero vector, but with the fifth entry set to 50. From this 
perturbed initial guess, Newton's method is seen to 
perform better than NA. Further, $\gnaarg{0.1}$ and 
$\gnaarg{0.5}$ outperform both Newton and NA, with 
$\gnaarg{0.5}$ converging in about half as many 
iterations as NA. This demonstrates the flexibility 
offered by $\gna$. That is, whether Newton 
or NA is the best choice for a particular problem 
and initial guess, $\gna$ is more agnostic to 
these choices, and can perform well in either 
case. There is still, however, sensitivity 
to $\hat{r}$.
We found that $\gna$ failed to converge for 
$\hat{r} = 0.3$, $0.5$, $0.6$, $0.7$, and $0.8$. $\gna$ converged with 
$\hat{r} = 0.4$, but only after 74 iterations. We  observed that 
$\gamma$NA(0.5), the non-adaptive version of $\gamma$-safeguarding, managed to 
converge. The reason for this variation is likely due to the complex behavior of Newton and NA 
in the preasymptotic regime, leading to sensitivity to the 
choice of $\hat{r}$. With preasymptotic 
safeguarding, our methods failed for $\mu < 0.94$, hence 
continuation may still be required in some cases,
but NA and $\gna$ can be used to efficiently solve the 
problem closer to the bifurcation point compared to 
Newton, thereby reducing the overall computational 
cost.
The point is that applying $\gamma$NA$(r)$ and $\gna$ in the 
preasymptotic regime can still lead to faster convergence 
than Newton, but this convergence is again 
sensitive to the choice of $\hat{r}$, and further
work is required to understand this sensitivity.

\begin{figure}[H]
\centering
\includegraphics[width=80mm]{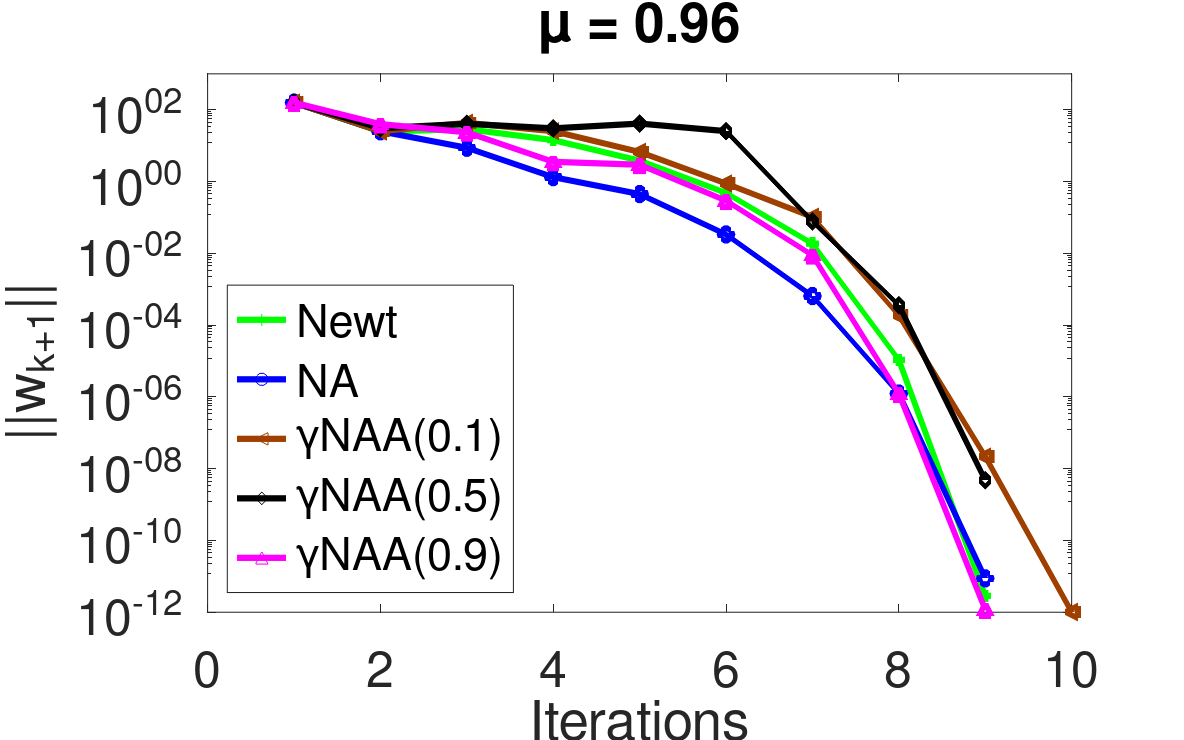}
\includegraphics[width=80mm]{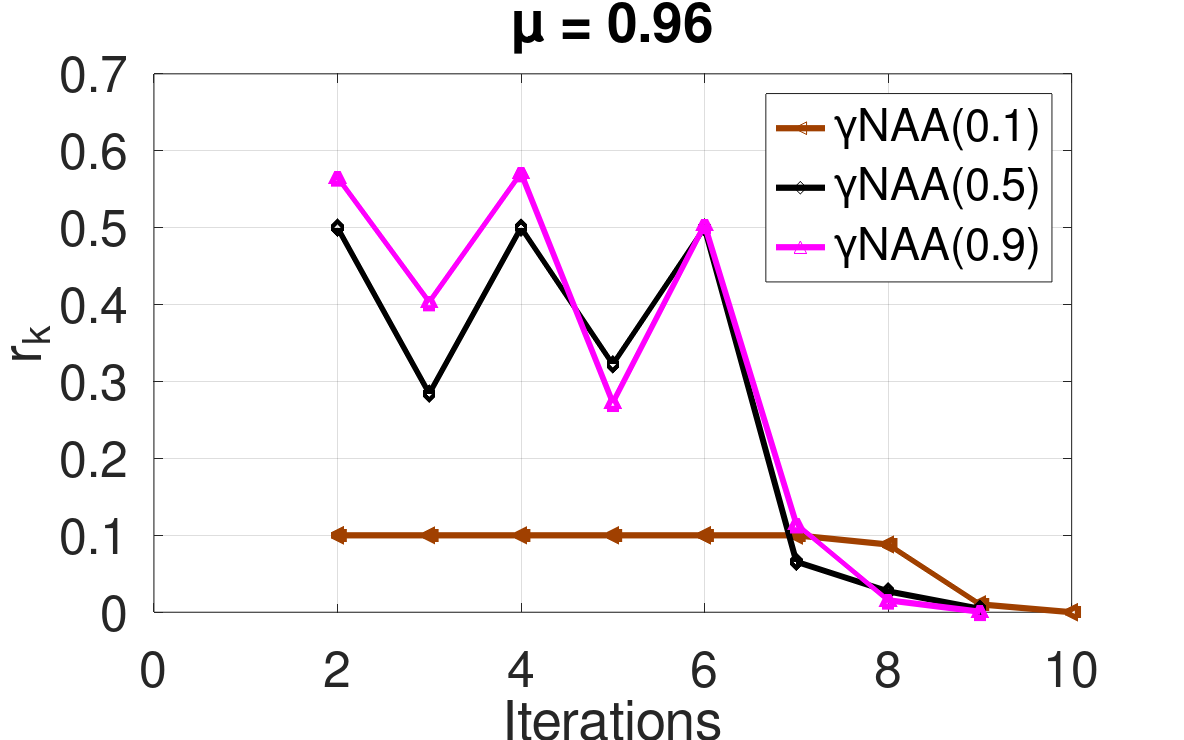}
\caption{ \label{fig:mu0.96}
Results of preasymptotic safeguarding applied to 
Model \eqref{coanda-model} for $\mu=0.96$. Left: Convergence history for Newton, NA, and 
$\gna$ for $\mu = 0.96$. Right: The value of 
$r_{k+1}$ at each iteration for the adaptive methods.}
\end{figure}

\begin{figure}[H]
\centering
\includegraphics[width=80mm]{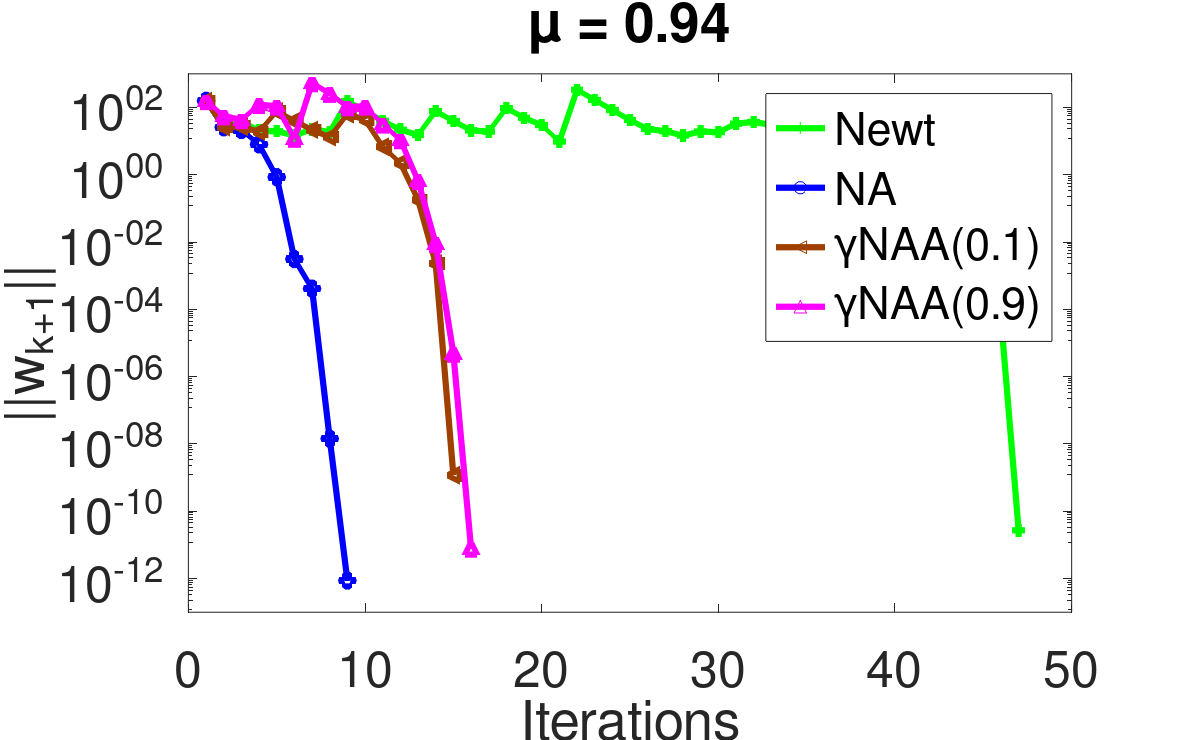}
\includegraphics[width=80mm]{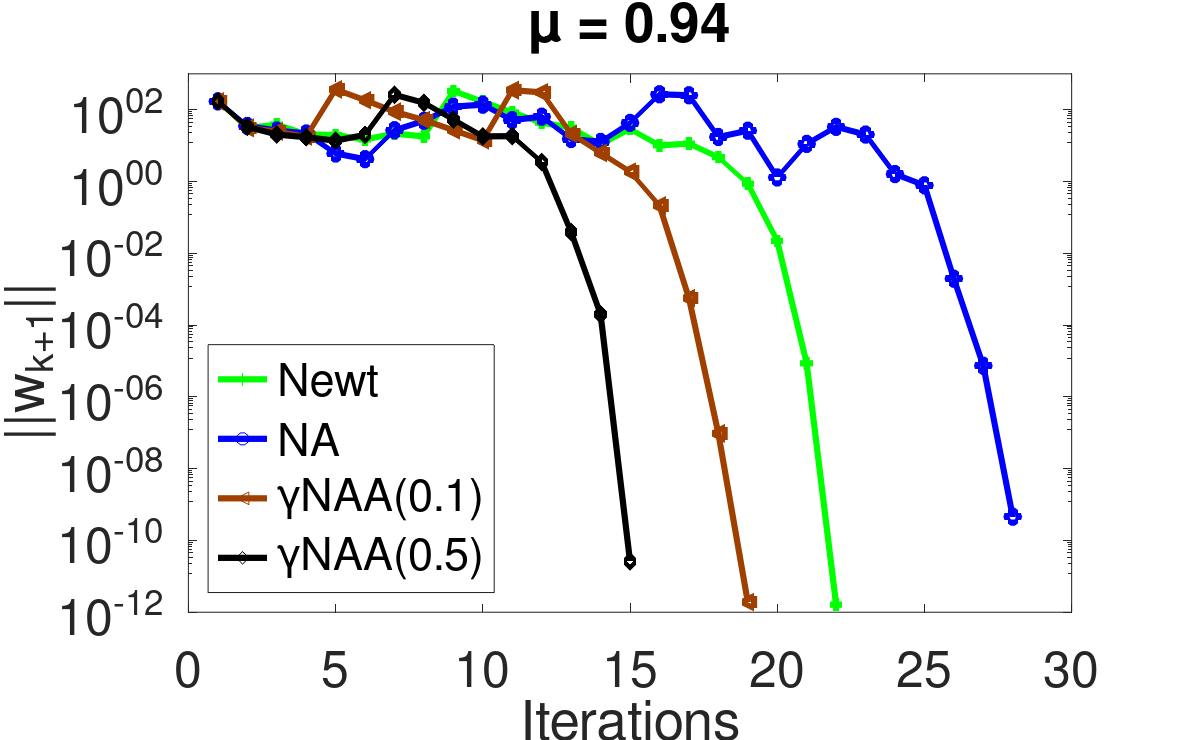}
\caption{\label{fig:mu0.94} Results of preasymptotic safeguarding applied to 
Model \eqref{coanda-model} for $\mu=0.94$. Left: Convergence history for Newton, NA, and 
$\gna$ for $\mu = 0.94$ with 
zero vector initial guess.
Right: Convergence history for Newton, NA, 
and $\gna$ for 
$\mu = 0.94$ with perturbed initial guess.}
\end{figure}

\subsubsection{Results for Rayleigh-B\'enard Model}

With preasymptotic safeguarding applied to 
Model \eqref{rayben-model}, we again see more varied
behavior since we have $\gamma$-safeguarding activated from 
the beginning of the solve. The results are shown in Figures \ref{fig:preasym-ri3.1}, \ref{fig:preasym-ri3.2}, and 
 \ref{fig:preasym-ri3.4}. The $r_{k+1}$ plots are only 
shown for those values of $\hat{r}$ for which $\gna$ converged.
We found that for $\ri = 3.0$ and $\ri = 3.1$, $\gna$ 
failed to converge with $\hat{r} = 0.1$, $0.5$, or 
$0.9$. Convergence is recovered with $\hat{r} = 0.4$ and 
$\hat{r} = 0.6$ respectively. We also ran $\gna$ with these
$\hat{r}$ values for $\ri = 3.2$, $3.3$, $3.4$, and $3.5$. The results were similar for 
$\ri = 3.3$, $3.4$, and $3.5$. Hence 
we only show the results for $\ri = 3.4$ 
in Figure \ref{fig:preasym-ri3.4}.
The theme demonstrated in Figures \ref{fig:preasym-ri3.1}, \ref{fig:preasym-ri3.2}, and 
 \ref{fig:preasym-ri3.4} is that 
when preasymptotic
safeguarding is employed, it is possible to converge faster
than standard NA. Moreover, as seen in Figure 
\ref{fig:preasym-ri3.2}, $\gna$ can converge when both Newton and NA diverge. One may note that in the right-most 
plot in Figure \ref{fig:preasym-ri3.1},
prior to the asymptotic regime where $r_{k+1}\to 0$, 
we observe $r_{k+1} <
\hat{r}=0.6$ only twice. Using the restarted Anderson 
interpretation discussed in Section \ref{sec:gendiscussion},
we could say that there are two {\it quasi-restarts} prior 
to the asymptotic regime. Evidently, these two 
quasi-restarts are essential, since we found that
non-adaptive $\gamma$-safeguarding with $\hat{r}=0.6$, 
Algorithm \ref{alg:gsg-og}, diverges. Similar
behavior of $r_{k+1}$ is seen in Figures 
\ref{fig:preasym-ri3.2} and \ref{fig:preasym-ri3.4}.
It remains unclear
precisely how the choice of $\hat{r}$ affects convergence, 
e.g., in Figure \ref{fig:preasym-ri3.1}, why does 
$\gnaarg{0.6}$ converge, but $\gna$ diverges for 
$\hat{r}=0.1$, $0.4$, $0.5$, and $0.9$? As 
previously discussed, a better 
understanding of these methods in the preasymptotic regime 
would help answer questions like these, and this will be
the focus of future studies. 

\begin{figure}[H]
\centering
\includegraphics[width=80mm]{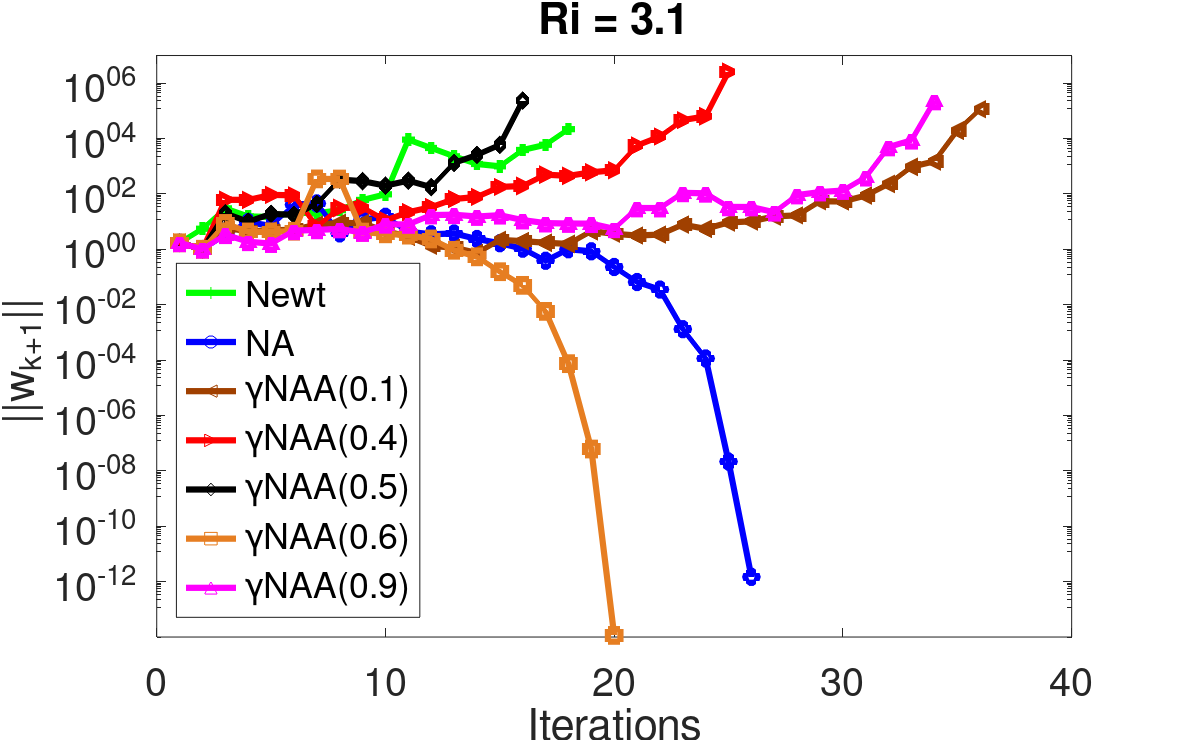}
\includegraphics[width=80mm]{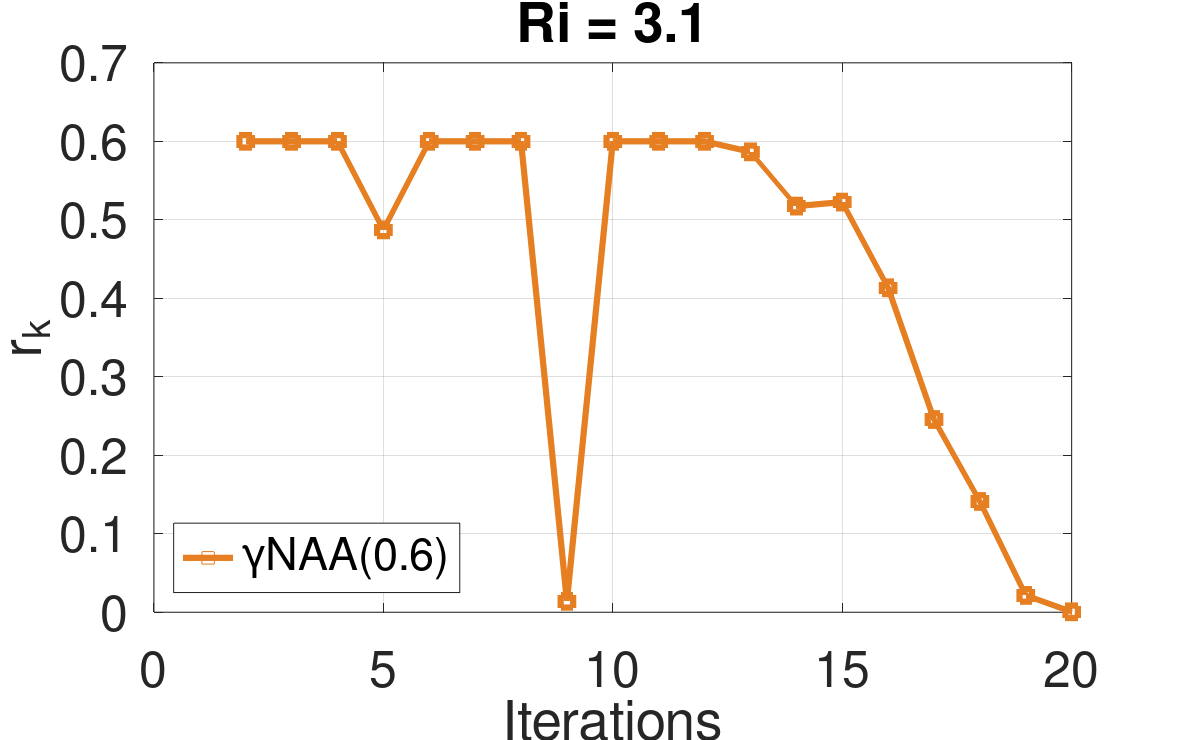}
\caption{ \label{fig:preasym-ri3.1} Comparison of 
$\gna$ applied preasymptotically with Newton and NA applied 
to Model \eqref{rayben-model} with $\ri = 3.1$.}
\end{figure}

\begin{figure}[H]
\centering
\includegraphics[width=80mm]{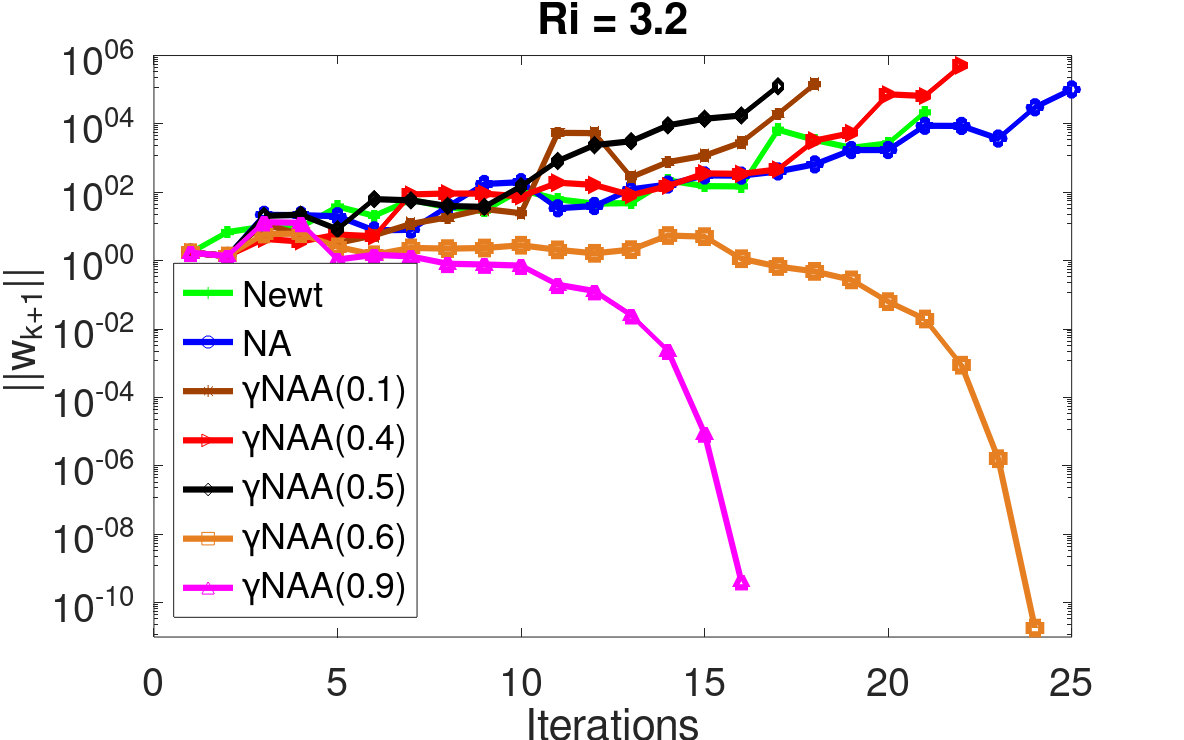}
\includegraphics[width=80mm]{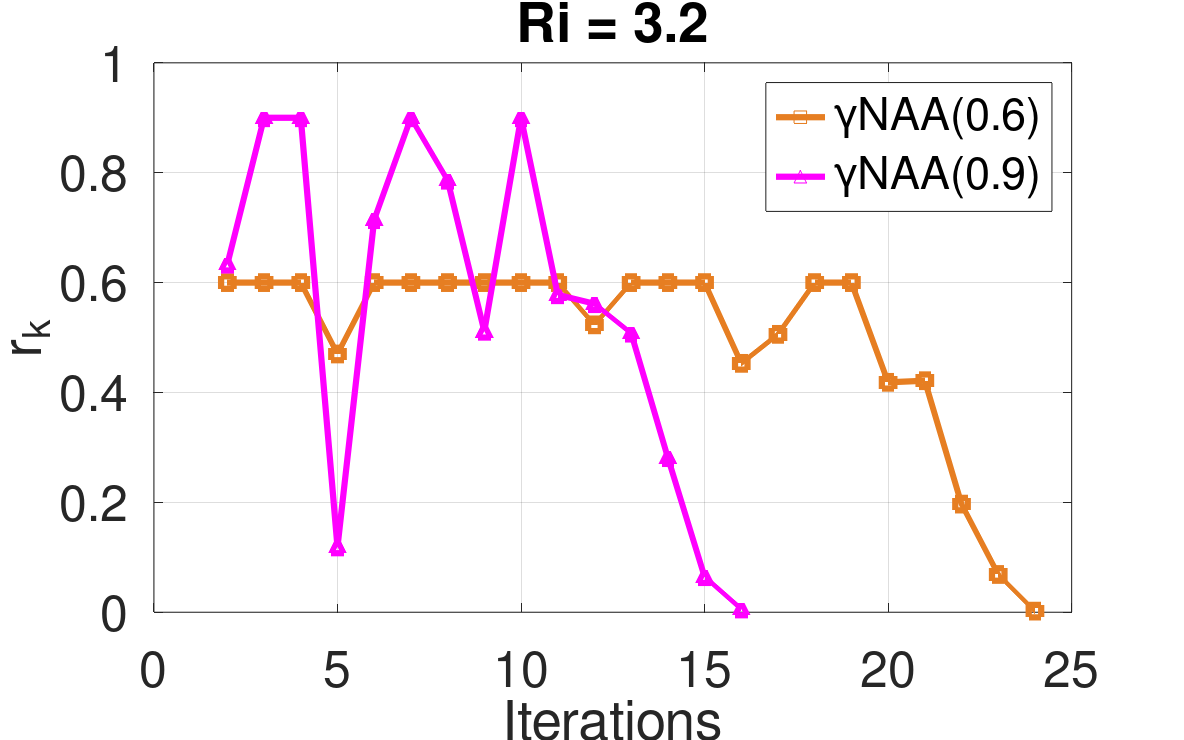}
\caption{ \label{fig:preasym-ri3.2} Comparison of 
$\gna$ applied preasymptotically with Newton and NA applied 
to Model \eqref{rayben-model} with $\ri = 3.2$.}
\end{figure}

\begin{figure}[H]
\centering
\includegraphics[width=80mm]{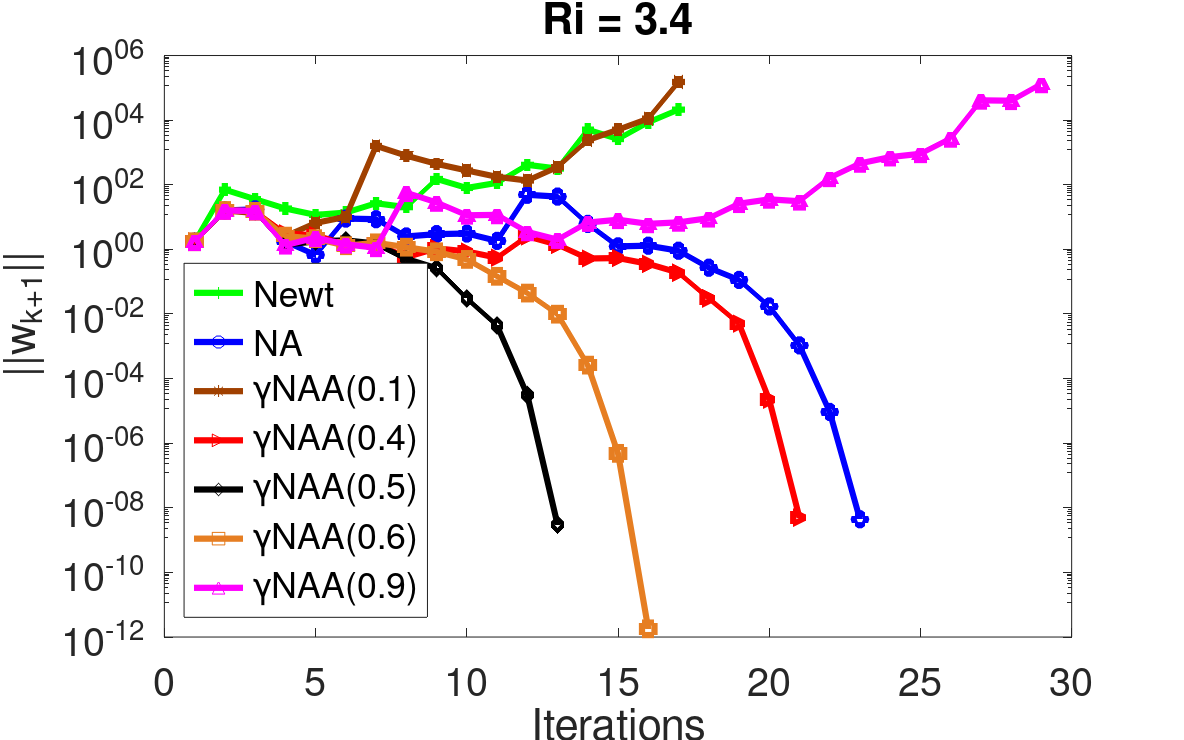}
\includegraphics[width=80mm]{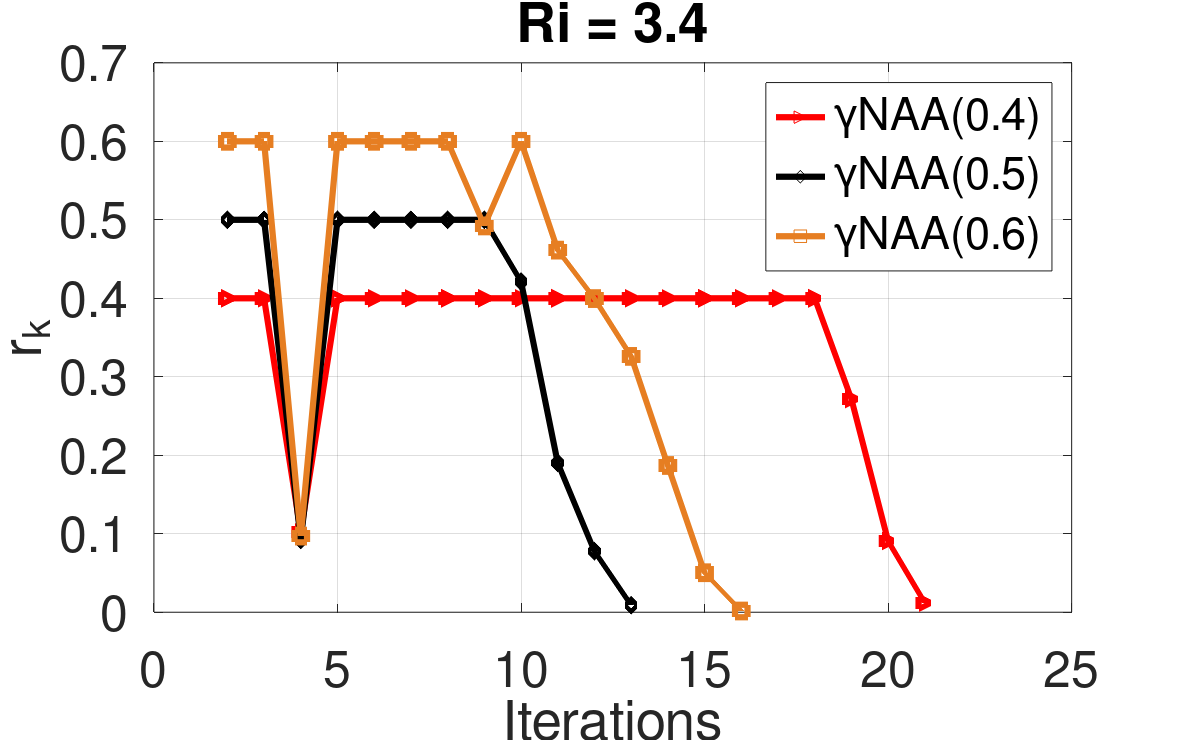}
\caption{ \label{fig:preasym-ri3.4} Comparison of 
$\gna$ applied preasymptotically with Newton and NA applied 
to Model \eqref{rayben-model} with $\ri = 3.4$.}
\end{figure}

\subsection{Increasing Anderson Depth}
\label{sec:incnewtdepth}

{ 
The strategy employed in this section is to increase the 
depth of the NA from $m=1$ used in previous sections.}

\subsubsection{Results for Channel Flow Model}
The plots in Figure \ref{fig:m3comp1} demonstrate the 
effectiveness of increasing the algorithmic depth $m$ 
to solve Model \eqref{coanda-model} near the bifurcation point. We also experimented with applying $\gna$ asymptotically. When $\|w_{k+1}\| < 1$, we set $m=1$ and activated $\gna$ 
with $\hat{r}=0.9$. 
These results are seen as the 
dashed lines in the left-most plot in Figure 
\ref{fig:m3comp1}. 
The philosophy is similar to that of 
preasymptotic safeguarding. We use NA(3) to reach the 
asymptotic regime, and then allow adaptive $\gamma$-safeguarding 
to detect if the problem is nonsingular. 
We set $m=1$ because, 
for the present, $\gamma$-safeguarding is only designed 
for $m=1$. The left-most plot in Figure \ref{fig:m3comp1} 
shows how, from the same initial iterate, we are able 
to solve Model \eqref{coanda-model} for a wider range of 
$\mu$ values, including in the regime where Newton, 
$\NA (1)$, and $\gna$ failed to converge.  

The right-most
plot in Figure \ref{fig:m3comp1} focuses on the 
results of applying $\NA (3)$ to Model \eqref{coanda-model} 
with $\mu = 0.92$. The point here is that even in the 
regime where $\NA (1)$ converges, $\NA (3)$ converges
 in about half as many iterations. Thus increasing 
the algorithmic depth can lead to faster convergence.
The improved convergence seen with increasing the depth $m$ suggests that 
a generalization of $\gamma$-safeguarding for greater depths could be 
useful as a generalization of the strategies presented in previous sections. 
However, for this strategy to be effective in general, 
further study is needed on the proper choice of depth $m$. 
In the chosen parameter regime, $m=3$ 
was the only value of $m\in\{1,2,...,10\}$ observed to improve convergence
when NA(1) failed.

\begin{figure}[H]
\centering
\includegraphics[width=80mm]{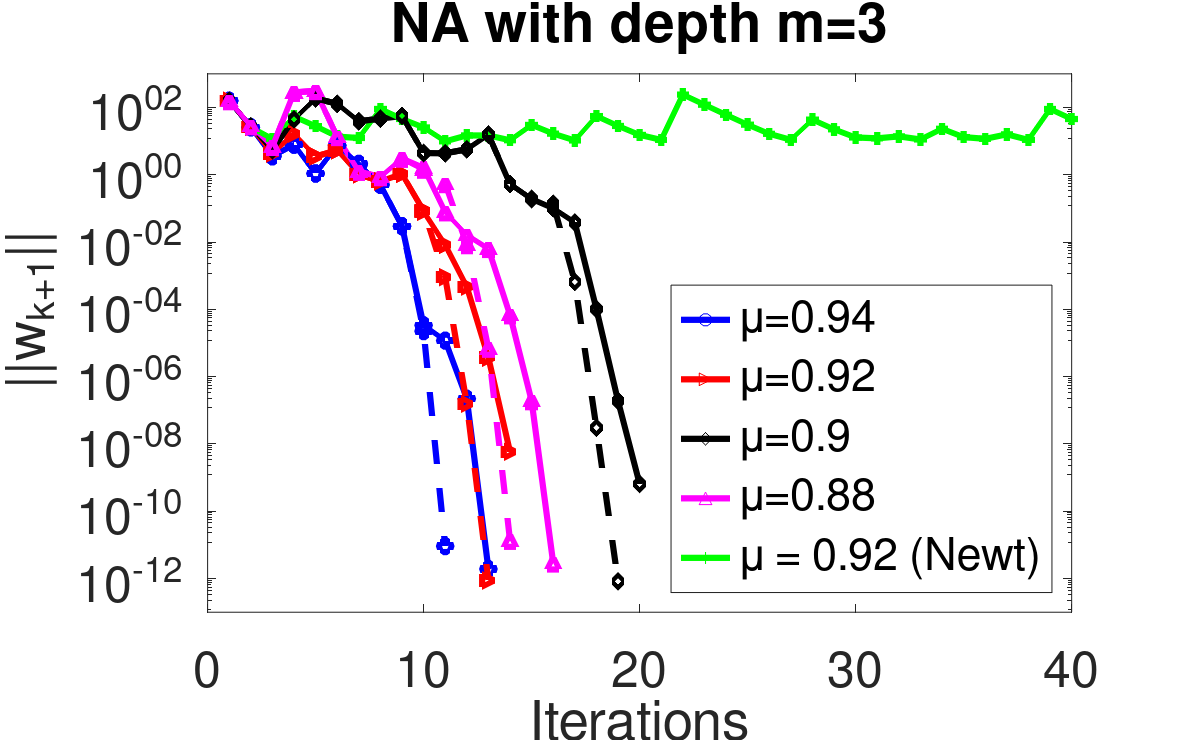}
\includegraphics[width=80mm]{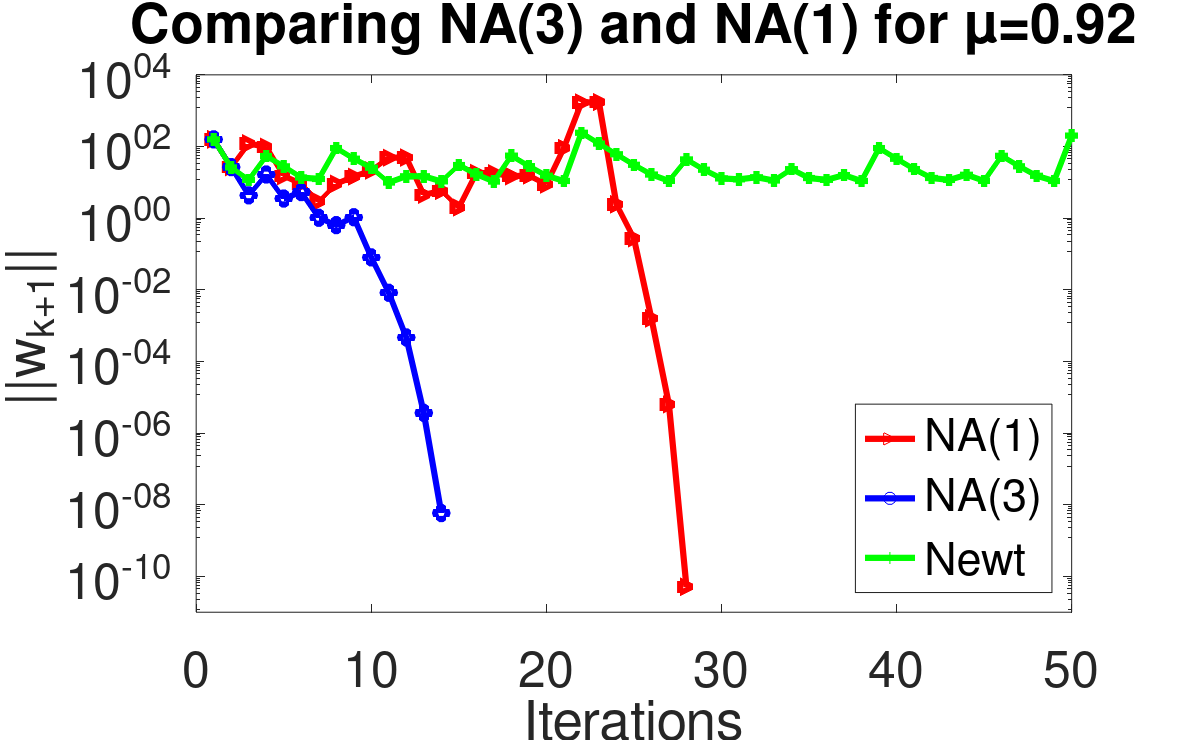}
\caption{\label{fig:m3comp1} Results from increasing Anderson 
depth $m$. Left: Convergence 
history of NA(3) for various choices of $\mu$ near $\mu^*$. Solid lines are NA$(3)$, and 
dashed lines are NA$(3)$ with asymptotic $\gna$. 
Right: Convergence history 
for NA(1) and NA(3) with $\mu=0.92$.}
\end{figure}

\subsubsection{Results for Rayleigh-B\'enard Model}
The results of increasing $m$ to solve Model \eqref{rayben-model}
are shown below in 
Figure \ref{fig:bouss_nam_compare}. 
As with Model \eqref{coanda-model}, we experimented with reducing 
$m=1$ and activating $\gamma$-safeguarding asymptotically. The dashed lines in Figure \ref{fig:bouss_nam_compare} are the results of these experiments.
For this problem, activation occurred when 
$\|w_{k+1}\| < 10^{-1}$ since we found that activating 
$\gna$ with $\hat{r}=0.9$ when $\|w_{k+1}\|<1$ broke convergence like it did 
in Section \ref{sec:rayben-asymsg}.
We once again 
observe that the right choice of $m > 1$
can improve convergence 
significantly. This can be seen in Figure \ref{fig:bouss_nam_compare}
for $\ri = 3.3$. Moreover, Figure \ref{fig:bouss_nam_compare} demonstrates that increasing $m$ can recover convergence when $\NA (1)$ 
fails to converge for $\ri = 3.0$ and $\ri = 3.2$. 
However, for $\ri =$ 3.1, 3.4, and 3.5, there was no significant improvement
 gained from increasing $m > 1$. The results for 
$\ri = 3.4$ shown below in Figure 
\ref{fig:bouss_nam_compare} are representative of the results for $\ri = 3.1$ and $\ri=3.5$. 

The most significant difference between 
$\NA (m)$ without asymptotic $\gna$, and $\NA (m)$ with asymptotic $\gna$, is  
seen with $m=10$.
It also appears that $m=10$ is more sensitive to the 
activation threshold than smaller choices of $m$. This is 
seen in Figure \ref{fig:bouss_nam_compare} for $\ri = 3.3$
and $\ri = 3.4$. These results are promising, and 
motivate further investigation to fully understand 
how the choice of $m$ affects convergence near singularities,
and in particular, near bifurcation points.

\begin{figure}[H]
\centering
\includegraphics[width=80mm]{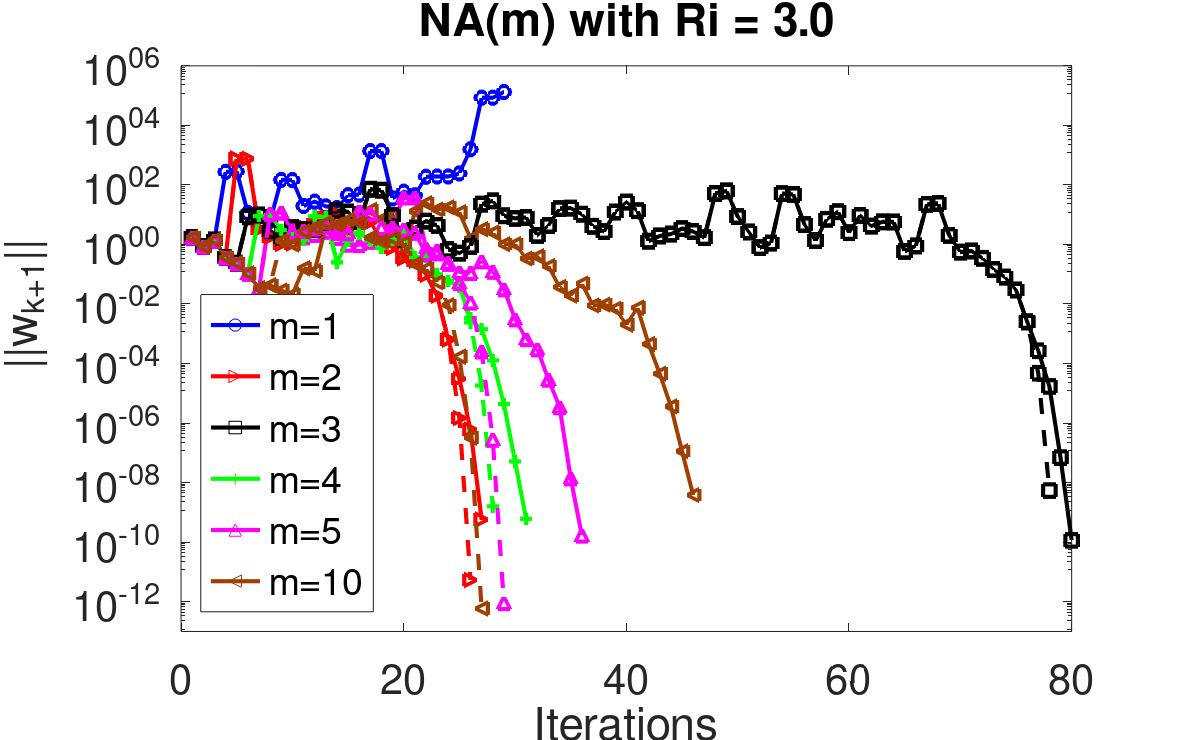}
\vspace{0.5em}
\includegraphics[width=80mm]{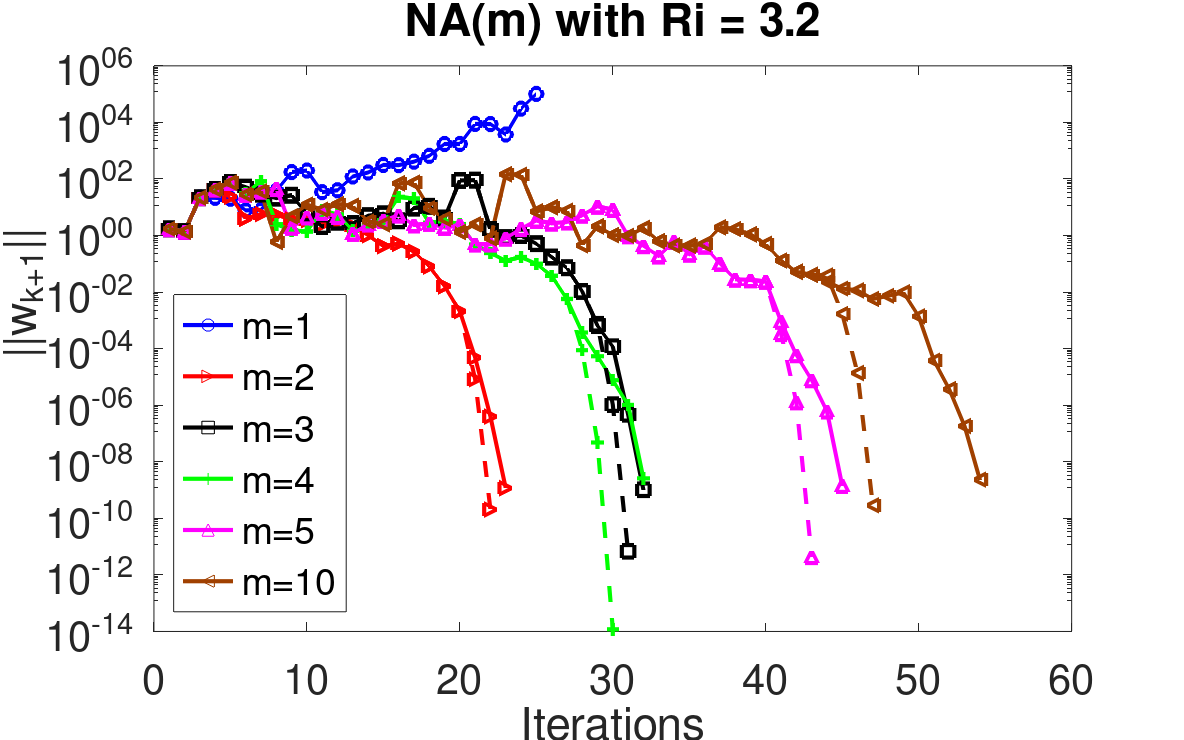}
\vspace{0.5em}
\includegraphics[width=80mm]{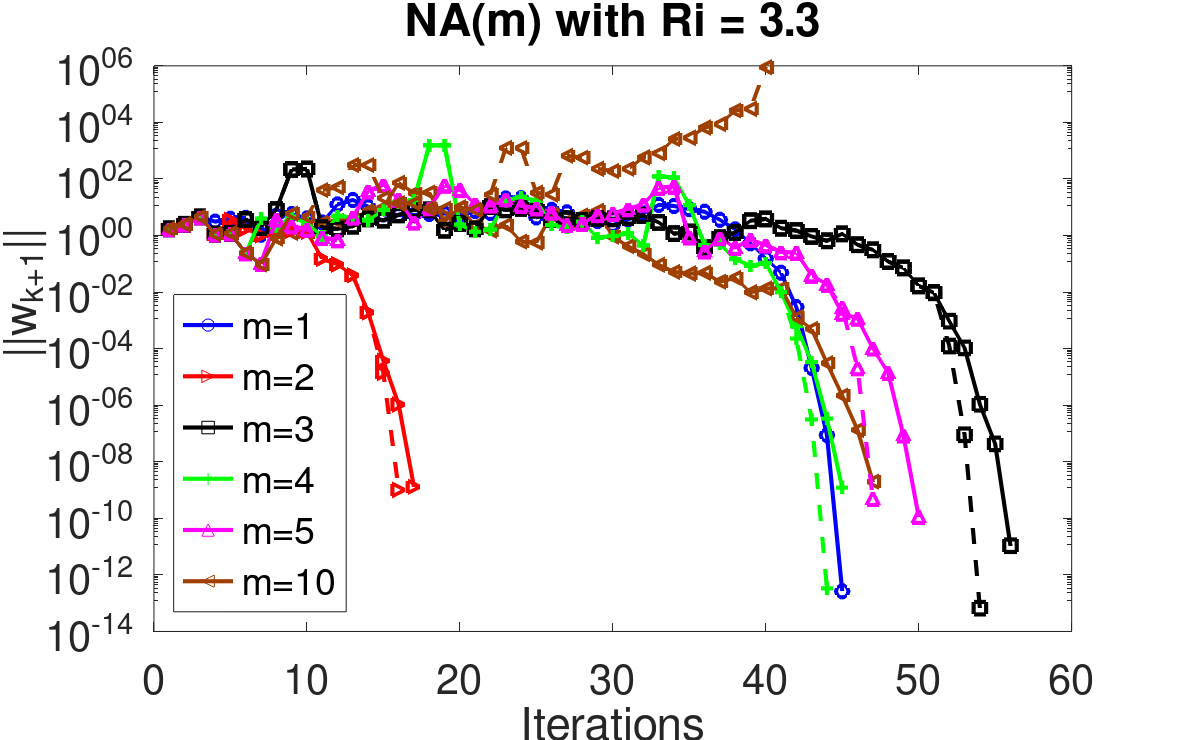}
\vspace{0.5em}
\includegraphics[width=80mm]{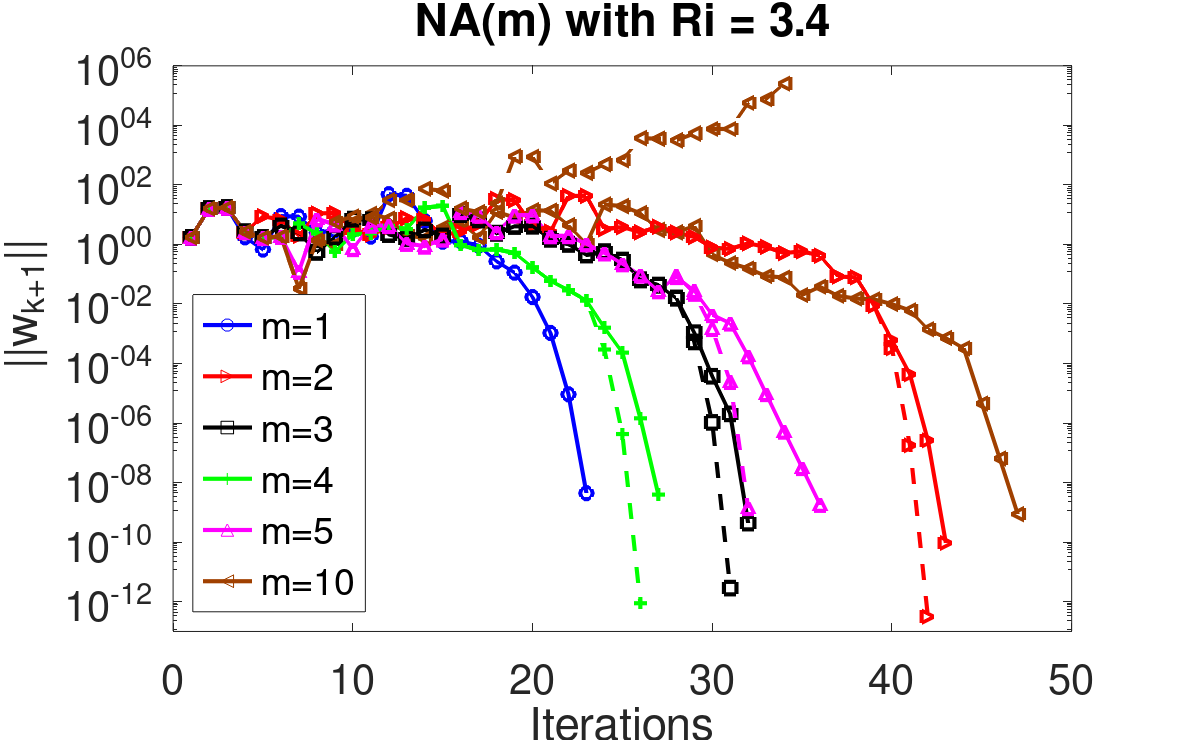}
\vspace{0.5em}
\caption{ \label{fig:bouss_nam_compare} Convergence history 
of $\NA (m)$ for $m=$ 1,2,3,4,5, and 10 for $\ri=$3.0, 3.2, 3.3, and 3.4. Dashed lines denote $\NA (m)$ with asymptotic 
$\gna$ with $\hat{r}=0.9$ activated when $\|w_{k+1}\|<10^{-1}$.}
\end{figure}

\section{Conclusion}
We have presented a modification of Anderson accelerated 
Newton's method for solving nonlinear equations near 
bifurcation points. We proved that, locally, this modified
scheme can detect nonsingular problems and scale the 
iterates towards a pure Newton step, which leads to faster
local convergence compared to standard NA. 
We numerically demonstrated two strategies one can 
employ when using our modified NA scheme 
to solve nonlinear problems near bifurcation points, with
our test problems being two Navier-Stokes type parameter-dependent PDEs. Asymptotic safeguarding was shown to recover
local quadratic convergence when the problem is nonsingular,
and it shows virtually no sensitivity to the choice of 
parameter $\hat{r}$. It can, however, be sensitive to the 
choice of activation threshold. Preasymptotic safeguarding 
is shown to significantly improve convergence, and 
can recover convergence when both Newton and $\NA$ fail. 
There is strong sensitivity to the choice of $\hat{r}$ though,
and future work will clarify this dependence.
We also demonstrated that increasing the Anderson depth
$m$ can improve convergence, and increase the domain of 
convergence with respect to the problem parameter. Future 
projects will further study how the choice of $m$ impacts 
convergence, and work towards developing $\gamma$-safeguarding 
for greater algorithmic depths. 

\section{Acknowledgements}

MD and SP are supported in part by the National Science 
Foundation under Grant No. DMS-2011519. LR is supported in 
part by the National Science Foundation under Grant No. 
DMS-2011490. This material is based upon work supported by 
the National Science Foundation under Grant No. DMS-1929284 
while the authors were in residence at the Institute for 
Computational and Experimental Research in Mathematics in 
Providence, RI, during the Acceleration and Extrapolation 
Methods (MD, SP and LR), and the Numerical PDEs: Analysis, 
Algorithms and Data Challenges (SP and LR), programs.

\bibliographystyle{plain}
\bibliography{app_to_bif_pts.bib}

\end{document}